\documentclass{amsart}

\usepackage{amsmath,latexsym}
\usepackage{amssymb}
\usepackage{graphics, graphicx}
\usepackage{amsfonts}
\usepackage[nesting]{hyperref}
\usepackage{longtable}
\usepackage{color}
\usepackage{pdflscape}

\newtheorem{theorem}{Theorem}

\newtheorem{defn}[theorem]{Definition}
\newtheorem{lemma}[theorem]{Lemma}

\newtheorem{proposition}[theorem]{Proposition}

\newtheorem{remark}[theorem]{Remark}
\newtheorem{ex}[theorem]{Example}

\usepackage[left=2.7cm, right=2.7cm, top=2.5cm, bottom=2.7cm]{geometry}

\usepackage{multicol}

\newtheorem*{putinar}{Putinar's Property}{\bf}{\it}

\newcommand{\dsum}{\displaystyle\sum}
\newcommand{\dprod}{\displaystyle\prod}
\def\K{\mathbf{K}}

\def\R{\mathbb{R}}
\def\N{\mathbb{N}}
\def\Q{\mathbf{Q}}

\def\B{\mathcal{B}}
\def\M{\mathcal{M}}
\def\P{\mathcal{P}}
\def\mrf{\mathbf{MOMRF}}
\def\om{\mathrm{OM}}
\def\Mo{\mathrm{M}}
\def\L{\mathrm{L}}

\def\ll{\ell}

\begin{document}
\title[Minimizing ordered weighted  averaging of rational functions]{A Semidefinite Programming approach for minimizing ordered weighted averages of rational functions}

\author{V\'ictor Blanco}
\address{Departamento de \'Algebra, Universidad de Granada}
\email{vblanco@ugr.es}

\author{Safae El-Haj-Ben-Ali \and Justo Puerto}
\address{Departamento de Estad\'istica e Investigaci\'on Operativa, Universidad de Sevilla}
\email{anasafae@gmail.com; puerto@us.es}


\date{}
\maketitle

\begin{abstract}
This paper considers the problem of minimizing the ordered weighted average (or ordered median) function of finitely many rational functions  over compact semi-algebraic sets. Ordered weighted averages of rational functions are not, in general, neither rational functions nor the supremum of rational functions so that current results available for the minimization of rational functions cannot be applied to handle these problems. We prove that the problem can be transformed into a new problem embedded in a higher dimension space where it admits a convenient representation. This reformulation admits  a hierarchy of SDP relaxations that approximates, up to any degree of accuracy, the optimal value of those problems. We apply this general framework to a broad family of continuous location problems showing that some difficult problems (convex and non-convex) that up to date could only  be solved on the plane and with Euclidean distance, can be reasonably solved with different $\ell _p$-norms and in any finite dimension space. We illustrate this methodology with some extensive computational results on location problems in the plane and the $3$-dimension space.

\end{abstract}

\section{Introduction}
Weighted Averaging (OWA) or Ordered Median Function (OMF) operators provide a parameterized class of mean type aggregation operators (see \cite{NP05,YaKa97} and the references therein for further details). Many notable mean operators such as the max, arithmetic average, median, k-centrum, range and min, are members of this class. They have been widely used in location theory and computational intelligence because of their ability to represent flexible models of modern logistics and linguistically expressed aggregation instructions in artificial intelligence (\cite{NP05} and \cite{Ya88,Ya2009,Ya2006,Ya2004,Ya1996,YaKa97}). Weighted averages (or ordered median) of rational functions are not, in general, neither rational functions nor the supremum of rational functions so that current results available for the minimization of rational functions are not applicable.
In spite of its intrinsic interest, as far as we know, a common approach for solving this family of problems is not available. Nevertheless, one can find in the literature different methods for solving particular instances of problems within this family, see e.g.  \cite{DN09-01,DN09-02,OgKr2005,NP05,OgSl10,Ogry2010,OgSl,OgSl03,OgZa02,OgRu02,PuTa05,ERChD10}.  The first goal of this paper is to develop a unified tool for solving this class of optimization problems. In this line,  we prove that the general problem can be transformed into a new problem embedded in a higher dimension space where it admits a convenient representation that allows to arbitrarily approximate or to solve it as a minimization problem over an adequate closed semi-algebraic set. Hence, our approach goes beyond a trivial adaptation of current theory.

Regarding the applications, it is commonly agreed that ordered median location problems are among the most important applications of OWA operators. Continuous location has achieved an important degree of
maturity. Witnesses of it are the large number of papers and
research books published within this field. In addition, this development
has been also recognized by the mathematical community
since the AMS code 90B85 is reserved for this area of
research. Continuous location problems appear very often in economic
models of distribution or logistics, in statistics when one
tries to find an estimator from a data set or in pure optimization
problems where one looks for the optimizer of a certain function.
For a comprehensive overview the reader is referred to \cite{DH02} or
\cite{NP05}. Despite the fact that many continuous location problems rely
heavily on a common framework, specific solution approaches
have been developed for each of the typical objective functions
in location theory (see for instance \cite{DH02}). To overcome this inflexibility and to work towards
a unified approach to location theory the so called Ordered
Median Problem (OMP) was developed (see \cite{NP05} and references
therein). Ordered median problems represent as special cases
nearly all classical objective functions in location theory, including
the Median, CentDian, center and k-centra. More precisely,
the 1-facility ordered median problem in the plane can be formulated
as follows: A vector of weights $(\lambda_1,\ldots,\lambda_n)$ is given. The
problem is to find a location for a facility that minimizes the
weighted sum of distances where the distance to the closest point
to the facility is multiplied by the weight $\lambda_n$, the distance to the second closest, by $\lambda_{n-1}$, and so on. The distance to the farthest point is multiplied by $\lambda_1$.
Many location problems can be formulated as the ordered 1-median
problem by selecting appropriate weights. For example,
the vector for which all $\lambda_i= 1$ is the unweighted 1-median problem,
the problem where $\lambda_n= 1$ and all others are equal to zero is
the 1-center problem, the problem where $\lambda_1=\ldots=\lambda_k=1$ and all others are equal to zero is the $k$-centrum. Minimizing the range of distances is
achieved by $\lambda_1=1$, $\lambda_n= -1$ and all others are zero. Despite its full generality, the main drawback of this framework is the difficulty of solving the problems with a unified tool. There have been some successful approaches that are now available whenever the framework space is either discrete (see \cite{BDNP05,MNPV08,OgZa02}) or a network (see \cite{KNP}, \cite{KNPT} or  \cite{NiPu99}). Nevertheless, the continuous case has been, so far, only partially covered. There have been some attempts to overcome this drawback and there are nowadays some available methodologies to tackle these problems, at least in the plane and with Euclidean norm. In Drezner \cite{drezner2007} and Drezner and Nickel \cite{DN09-01,DN09-02} the authors present two different approaches. The first one uses a continuous branch and bound method based on triangulations (BTST) and the second one on a D-C decomposition for the objective function that allow solving the problems on the plane. More recently, Rodriguez-Chia et al. \cite{ERChD10} also address the particular case of the $k$-centrum problem and using geometric arguments develop a better algorithm applicable only for that problem on the plane and Euclidean distances.

Quoting the conclusions of the authors of \cite{DN09-01}: \textit{``All our experiments were conducted for Euclidean distances. As
future research we suggest to test these algorithms on problems
(even the same  problems) based on other distance measures. (...)
 Solving k-dimensional problems by a similar
approach requires the construction of k-dimensional Voronoi
diagrams which is extremely complicated.''}

Therefore, the challenge is to design a common approach also to solve the above mentioned family of location problems, for different distances and in any finite dimension. This is essentially the second goal of this paper. In our way, we have addressed the more general problem that consists of the minimization of the OWA operator of a finite number of rational functions over closed semialgebraic sets that is the first goal of this paper. Thus, our second goal is to solve a general class of continuous location problems using the general approach mentioned above for the minimization of OWA rational functions and  to show the powerfulness of this methodology. Of course, we know that the problem in its full generality is $NP-hard$ since it includes general instances of convex minimization. Therefore, we cannot expect to obtain polynomial algorithms for this class of problems. Rather, we will apply a new methodology first proposed by Lasserre \cite{lasserre1}, that provides a hierarchy of semidefinite problems that converge to the optimal solution of the original problem, with the property that each auxiliary problem in the process can be solved in polynomial time.

%\subsection{Contribution}
The paper is organized in 5 sections. The first one is our introduction.
In the second section and for the sake of completeness, we recall some general results on the Theory of Moments and Semidefinite Programming (SDP) that will be useful in the rest of the paper. Section \ref{s:omrf} considers what we call the $\mathbf{MOMRF}$ problem
which consists of minimizing the \textit{ordered median function} of finitely many rational functions over a
compact basic semi-algebraic set.
In the spirit of the moment approach
developed in Lasserre \cite{lasserre1,lasserre2} for polynomial
optimization and later adapted by Jibetean and De Klerk \cite{jibetean}, we define a hierarchy of semidefinite relaxations (in short
SDP relaxations). Each SDP relaxation is a semidefinite program
which, up to arbitrary (but fixed) precision,
can be solved in polynomial time
and the monotone sequence of optimal
values associated with the hierarchy converges to the optimal value of $%
\mathbf{MOMRF}$. Sometimes the convergence is finite and a sufficient
condition permits to detect whether a certain relaxation in the hierarchy is
exact (i.e. provides the optimal value), and to extract optimal solutions (theoretical bounds on the relaxation order for the exact results can be found in \cite{Schweighofer,markus}).
Section \ref{s:locomf} considers a general family of location problems that is built from  the problem $\mathbf{MOMRF}$ but which does not actually fits under the same formulation because the objective functions are not  quotients %ratio
of polynomials. Nevertheless, we prove that under a certain reformulation one can define another hierarchy of SDP that fulfils convergence properties `$\grave{\mbox{a}}$ la Lasserre'. This approach is applicable to location problems with any $\ell_p$-norm ($p\in \mathbb{Q}$) and in any finite dimension space. We exploit the special structure of these problems to find a block diagonal reformulation that reduces the sizes of the SDP relaxations and allows to solve larger instances. Our computational tests are presented in Section 5. We analyze five families of problems, namely, Weber, center, $k$-centrum, trimmed-mean and range. There we show that convergence is rather fast and very high accuracy is achieved in all cases, even with the first feasible relaxation. (We observe that for location problems with Euclidean distances that relaxation order is $r=2$.) The paper ends with some conclusions and an outlook for further research.

\section{Preliminaries}

In this section we recall the main definitions and results on the moment problem and semidefinite programming that will be useful for the development through this paper. We use standard notation in the field (see e.g. \cite{lasserrebook}).

We denote by $\mathbb{R}[x]$ the ring of real polynomials in the variables $x=(x_1,\ldots,x_n)$, and by $\mathbb{R}[x]_d \subset \R[x]$ the space of polynomials of degree at most $d \in \N$ (here $\N$ denotes the set of nonnegative integers). We also denote by $\B = \{x^\alpha: \alpha\in\mathbb{N}^n\}$ a
canonical basis of monomials for $\R[x]$, where $x^\alpha = x_1^{\alpha_1} \cdots x_n^{\alpha_n}$, for any $\alpha \in \N^n$.

%Let , and

%Finally, let $\Vert x\Vert_{p}$ denote the $\ell_p$-norm of $x\in\R^n$, namely $\|x\|_{p}=(\sum_{i=1}^n|x_i|^p)^{1/p}$.

For any sequence indexed in
the canonical monomial basis $\B$, $\mathbf{y}=(y_\alpha)_{\alpha \in \N^n}\subset\mathbb{R}$, let $\L_\mathbf{y}:\mathbb{R}[x]\to\mathbb{R}$ be the linear functional defined, for any $f=\sum_{\alpha\in\mathbb{N}^n}f_\alpha\,x^\alpha \in \R[x]$, as $\L_\mathbf{y}(f) :=
\sum_{\alpha\in\mathbb{N}^n}f_\alpha\,y_\alpha$.

The \textit{moment} matrix $\Mo_d(\mathbf{y})$ of order $d$ associated with $\mathbf{y}$, has its rows and
columns indexed by $(x^\alpha)$ and $\Mo_d(\mathbf{y})(\alpha,\beta)\,:=\,\L_\mathbf{y}(x^{\alpha+\beta})\,=\,y_{%
\alpha+\beta}$, for $\vert\alpha\vert,\,\vert\beta\vert\,\leq d.$ Note that the moment matrix is ${{n+d} \choose {n} }\times {{n+d} \choose {n}}$ and that there are ${{n+2d} \choose {n}}$ $\mathbf{y}_\alpha$ variables.

For $g\in \mathbb{R}[x] \,(=\sum_\gamma g_{\gamma\in \N^n} x^\gamma$), the \textit{localizing%
} matrix $\Mo_d(g, \mathbf{y})$ of order $d$ associated with $\mathbf{y}$
and $g$, has its rows and columns indexed by $(x^\alpha)$ and $\Mo_d(g,\mathbf{y})(\alpha,\beta):=\L_\mathbf{y}(x^{\alpha+\beta}%
g(x))=
\sum_{\gamma}g_\gamma y_{\gamma+\alpha+\beta}$, for $\vert\alpha\vert,\vert\beta\vert\,\leq d$.

\begin{defn}
Let $\mathbf{y}=(y_\alpha)\subset\mathbb{R}$ be a sequence indexed in
the canonical monomial basis $\B$. We say that $\mathbf{y}$ has a \textit{%
representing} measure supported on a set $\mathbf{K} \subseteq \R^n$ if there is some finite
Borel measure $\mu$ on $\mathbf{K}$ such that
$$
y_\alpha\,=\,\int_\mathbf{K} x^\alpha\,d\mu(x), \mbox{ for all } \alpha \in \R^n.
$$
\end{defn}

The main assumption that is needed to impose when one wants to assure the convergence of the SDP relaxations for solving polynomial optimization problems (see for instance \cite{lasserre22,lasserrebook}) was introduced by Putinar \cite{putinar} and it is stated as follows.
\begin{putinar}
\label{defput}
Let $\{g_{1}, \ldots, g_l\} \subset \mathbb{R}[x]$ and $\mathbf{K}:=\{x\in \mathbb{R}^{d}: g_{j}(x)\geq
0,:j=1,\ldots ,\ll \}$ a basic
closed semialgebraic set. Then, $\mathbf{K}$ satisfies Putinar's property if there exists $u\in
\mathbb{R}[x]$ such that:
\begin{enumerate}
\item $\{x: u(x)\geq 0\} \subset \R^n$ is compact, and
\item\label{putrep} $u\,=\,\sigma _{0}+\sum_{j=1}^{\ll}\sigma _{j}\,g_{j}$, for some  $\sigma_1, \ldots, \sigma_l \in \Sigma \lbrack x]$. (This expression is usually called a \textbf{Putinar's representation} of $u$ over $\mathbf{K}$).
\end{enumerate}
Being $\Sigma[x]\subset\mathbb{R}[x]$  the subset of polynomials that are sums of squares.
\end{putinar}
Note that Putinar's property is equivalent to impose that the quadratic polynomial $M-\sum_{i=1}^n x_i^2$ has a Putinar's representation over $\mathbf{K}$.

We observe that Putinar's property implies compactness of $\mathbf{K}$. It is easy to see that Putinar's property holds if either $\{x:g_{j}(x)\geq 0\}$ is compact for some $j$, or all $g_{j}$ are affine and $\mathbf{K}$ is compact. Furthermore, Putinar's property is not restrictive at all, since any semialgebraic  set $\mathbf{K}$ for which is known that $\sum_{i=1}^n x_i^2 \leq M$ holds for some $M>0$ and for all $x\in \mathbf{K}$, $\mathbf{K} = \mathbf{K} \cup \{g_{l+1}(x):=M-\sum_{i=1}^n x_i^2\geq
0\}$ verifies  Putinar's property.

The importance of Putinar's property
stems from the following result:

\begin{theorem}[Putinar \protect\cite{putinar}]
\label{thput} Let $\{g_{1}, \ldots, g_l\} \subset \mathbb{R}[x]$ and $\mathbf{K}:=\{x\in \mathbb{R}^{d}: g_{j}(x)\geq
0,:j=1,\ldots ,\ll \}$ satisfying Putinar's property. Then:
\begin{enumerate}
\item Any $f\in \mathbb{R}[x]$ which is strictly positive on $\mathbf{K}$ has a Putinar's representation over $\mathbf{K}$.
\item $\mathbf{y}=(y_{\alpha })$ has a representing measure on $\mathbf{K}$ if and only if $\Mo_{d}(\mathbf{y})\succeq 0$, and $\Mo_{d}(g_{j},\mathbf{y})\succeq 0$, for all $j=1,\ldots ,l$ and  $d\in \N$. %  \label{put-moments}
\end{enumerate}
(Here, the symbol $\succeq 0$ stands for semidefinite positive matrix.)
\end{theorem}

The following result that appears in \cite{jibetean} and \cite{laraki-lasserre} will be also important for the development in the next sections.

\begin{lemma}
\label{lemm-frac} Let $\mathbf{K}\subset \mathbb{R}^{d}$ be compact and let $%
p,q$ be continuous with $q>0$ on $\mathbf{K}$. Let $\M(\mathbf{K})$ be the set of finite Borel measures on $\mathbf{K}$ and let $\P(\mathbf{K})\subset \M(\mathbf{K})$ be its subset of probability measures on $\mathbf{K}$. Then
$$
\min_{\mu \in \P(\mathbf{K})}\frac{\int_\K p\,d\mu }{\int_\K q\,d\mu }
=\min_{\varphi \in \M(\mathbf{K})}\{\int_\K p\,d\varphi :\int_\K q\,d\varphi
\,=\,1\} = \min_{\mu \in \P(\mathbf{K})}\int_\K \frac{p}{q}\,d\mu \,=\,\min_{x\in
\mathbf{K}}\frac{p(x)}{q(x)}.
$$
\end{lemma}

\section{Minimizing the ordered weighted average of finitely many rational functions \label{s:omrf}}

Let $\mathbf{K}\subset \mathbb{R}^{d}$ be a basic semi-algebraic set defined as
$$
\mathbf{K}\,:=\,\{x\in \mathbb{R}^{d}:g_{j}(x)\geq 0,\quad j=1,\ldots ,\ell\}
$$
for $g_1, \ldots, g_{\ell} \in \R[x]$.

Let us introduce the function $\om(x) = \sum_{k=1}^{m} \lambda_{k}(x)f_{(k)}(x)$, for some rational functions $(f_{j})\subset \mathbb{R}[x]$, being $%
f_{k}=p_{k}/q_{k}$ rational functions with $%
p_{k},q_{k}\in \mathbb{R}[x]$, $\lambda_{k}(x) \in \R[x]$, and $f_{(k)}(x)\in \{f_{1}(x), \ldots, f_{m}(x)\}$ such that
 $f_{(1)}(x) \geq f_{(2)}(x) \geq \cdots \geq f_{(m)}(x)$ for $x \in \R^n$. We assume that $\mathbf{K}$ satisfies Putinar's property and that $q_k>0$ on $\mathbf{K}$, for every $k=1, \ldots, m$.

%\begin{itemize}
%\item $\K$ satisfies Putinar's property and,
%\item $q_{k}>0$ on $\K$ for every $k=1, \ldots, m$.
%\end{itemize}

Consider the following problem:
\begin{equation}
%\mrf: \quad
\rho_\lambda \,:=\displaystyle\min_{x}\{\om(x): x\in \mathbf{K}\,\}, \label{pro:omrp1}\tag{${\rm OMRP}_\lambda^0$}
\end{equation}%

Associated with the above problem we introduce an auxiliary problem.
For each $i = 1, \ldots, m$, $j = 1, \ldots, m$ consider the decision variables $w_{ij}$ that model for each $x\in \mathbf{K}$
$$
w_{ij} = \left\{\begin{array}{rl} 1 & \mbox{if $f_{i}(x)=f_{(j)}(x)$,}\\
0 & \mbox{ otherwise.}
\end{array}\right..
$$
Now, we consider the problem:
\begin{align}
&\quad \overline{\rho}_\lambda = &\displaystyle \min_{x,w} & \sum_{j=1}^{m}\lambda_{j}(x)\sum_{i=1}^{m} f_{i}(x)w_{ij}\label{omrp}\tag{${\rm OMRP}_\lambda$} \\
& &\mbox{s.t. } &\sum_{j=1}^{m}w_{ij} = 1, \text{ for } i=1, \ldots, m, \label{cons:first}\\
& & &\sum_{i=1}^{m}w_{ij} = 1, \text{ for } j=1, \ldots, m,\nonumber\\
& & & w_{ij}^2 - w_{ij} = 0, \text{ for } i,j=1, \ldots, m,\nonumber\\
& & &\sum_{i=1}^{m} w_{ij} f_{i}(x) \geq  \sum_{i=1}^{m} w_{i j+1} f_{i}(x),\; j=1, \ldots, m, \label{cons:3.4}\\
& & &\sum_{i=1}^m \sum_{j=1}^m w^2_{ij} \leq m, \label{const:3.5}\\
& & &w_{ij} \in \R, \text{ for } i,j=1, \ldots, m,\quad x\in \mathbf{K}. \label{cons:3.9}
\end{align}

The first set of constraints ensures that for each $x$,  $f_i(x)$ is sorted in a unique position. The second set ensures that the $j^{th}$ position is only assigned to one rational function.  The next constraints are added to assure that $w_{ij} \in \{0,1\}$. The fourth one states that
$f_{(1)}(x) \geq \cdots \geq f_{(m)}(x)$. The last set of constraints ensures the satisfaction of Putinar's property
of the new feasible region. (Note that this last set of constraints are redundant but it is convenient to add them for a better description of the feasible set.)

These two problems, \eqref{pro:omrp1} and \eqref{omrp} satisfy the following relationship.
\begin{theorem} \label{th:feas}
Let $x$ be a feasible solution of \eqref{pro:omrp1} then there exists a solution $(x,w)$ for \eqref{omrp} such that their objective values are equal. Conversely, if $(x,w)$ is a feasible solution for \eqref{omrp} then there exists a solution $x$ for \eqref{pro:omrp1} having the same objective value. In  particular $\varrho_{\lambda}=\hat{\varrho}_{\lambda}$.
\end{theorem}
\begin{proof}
Let $\bar x$ be a feasible solution of \eqref{pro:omrp1}. Then, it clearly satisfies that $\bar x\in \mathbf{K}$. In addition, let $\sigma$ be the permutation of $(1,\ldots,m)$ such that $f_{\sigma(1)}(\bar x)\ge  f_{\sigma(2)}(\bar x) \ge \ldots \ge f_{\sigma(m)}(\bar x)$. Take,
$$ \overline w_{ij}=\left\{\begin{array}{ll} 1 & \mbox{if } i=\sigma(j), \\
0 & \mbox{otherwise}. \end{array} \right. $$
Clearly, $(\bar x,\overline w)$ satisfy the constraints in (\ref{cons:first}-\ref{cons:3.9}). Indeed, for any $i$ %there exist $j=\sigma^{-1}(i)$ such that
 $\sum_{j=1}^m \overline w_{ij}= \overline w_{i\sigma^{-1}(i)}=1$. Analogously, for any $j$, $\sum_{i=1}^m \overline w_{ij}= \overline w_{\sigma(j),j}=1$. By its own definition, $\overline w$ only takes $0,1$ values and thus, $\overline w_{ij}^2-\overline w_{ij}=0$ for all $i,j$ and $\sum_{i,j}^m \overline w^2_{ij}\le m$.
Finally, to prove that $(\overline x, \overline w)$ satisfies (\ref{cons:3.4}), we observe, w.l.o.g.,  that for any $j$ there exist $i^*$ and $\hat i$ such that $\sigma(j)=i^*$ and $\sigma(j+1)=\hat i$. Hence,:
$$ \sum_{i=1}^m \overline w_{ij} f_i(\bar x)=\overline w_{i^* j} f_{\sigma(j)}(\bar x) \ge
\overline w_{\hat i j+1} f_{\sigma(j+1)}(\bar x)=\sum_{i=1}^m \overline w_{ij+1} f_i(\bar x).$$
 Moreover,$$OM_{\lambda}(\bar x)=
\sum_{j=1}^{m}\lambda_{j}(\overline x)\sum_{i=1}^{m} f_{i}(\bar x)\overline w_{ij}.$$

Conversely, if $(\bar x,\overline w)$ is a feasible solution of \eqref{omrp} then, clearly $\bar x$ is feasible of \eqref{pro:omrp1} and by the above,
$OM_{\lambda}(\bar x)=
\sum_{j=1}^{m}\lambda_{j}(x)\sum_{i=1}^{m} f_{i}(\bar x)\overline w_{ij}.$
\end{proof}

Then, we observe that  $f_i=p_i/q_i$ for each $i=1, \ldots, m$. Therefore, the constraint $\sum_{i=1}^{m} w_{ij} f_{i}(x) \geq  \sum_{j=1}^{m} w_{i j+1} f_{i}(x)$ can be written as a polynomial
constraint as
$$
\dsum_{i=1}^m w_{ij}\,p_i(x)\dprod_{k\neq i}^m q_k(x) \leq \dsum_{i=1}^m w_{i j+1}\,p_i(x)\dprod_{k\neq i}^m q_k(x) \quad j=1, \ldots, m.
$$
Let us denote by $\overline{\mathbf{K}}$ the basic closed semi-algebraic set that defines the feasible region of \eqref{omrp}.
\begin{lemma}
If $\mathbf{K}\subset \mathbb{R}^m$ satisfies Putinar's property then $\overline{\mathbf{K}}\subset \mathbb{R}^{n+m^2}$ satisfies  Putinar's property.
\end{lemma}
\begin{proof}
Since $\mathbf{K}$ satisfies Putinar's property,
 the quadratic polynomial $x\mapsto u(x):=M-\Vert x\Vert ^{2}_2$ can be
written as
$u(x)=\sigma_0(x)+\sum_{j=1}^p\sigma_j(x)g_j(x)$ for some s.o.s. polynomials $(\sigma_j)\subset\Sigma[x]$. Next,
consider the  polynomial
\begin{equation*}
(x,w)\mapsto r(x,w)\,=\,M+m-\Vert x\Vert ^{2}_2 - \sum_{i=1}^m\sum_{j=1}^m w^2_{ij}.
\end{equation*}%
Obviously, its level set $\{(x,w)\in \mathbb{R}^{n\times m^2}:r(x,z)\geq 0\}\subset \mathbb{R}^{n+m^2}$ is
compact and moreover, $r$ can be written in the form
\begin{equation*}
r(x,w)=\sigma _{0}(x)+\sum_{j=1}^{p}\sigma
_{j}(x)\,g_{j}(x)+1 \times \stackrel{\overline g(x,w) \; defining \; \overline K}{\overbrace{(m-\sum_{i=1}^m\sum_{j=1}^m w^2_{ij})}},
\end{equation*}
for appropriate s.o.s. polynomials $(\sigma'_j)\subset\Sigma[x,w]$. Therefore $\overline{\mathbf{K}}$
satisfies Putinar's property, the desired result.

\end{proof}

Now, we observe that the objective function of \eqref{omrp} can be written as a quotient of polynomials in $\R[x,w]$. Indeed, take
\begin{equation} \label{eq:PyQ} p_\lambda(x,w)= \sum_{j=1}^m \lambda_j(x) \dsum_{i=1}^m w_{ij}\,p_i(x)\dprod_{k\neq i}^m q_k(x) \mbox{ and } q_\lambda(x,w)= \dprod_{k= 1}^m q_k(x).
\end{equation}
\mbox{Then,}
\begin{equation}\label{eq:pdivq}
 \sum_{j=1}^{m}\lambda_{j}(x)\sum_{i=1}^{m} f_{i}(x)w_{ij}=\frac{p_\lambda(x,w)}{q_\lambda(x,w)}.
\end{equation}

Then, we can transform Problem \eqref{omrp} in an infinite dimension linear program on the space of Borel measures defined on $\mathbf{\overline K}$.

\begin{proposition}
\label{prop1} Let $\mathbf{\overline K}\subset \mathbb{R}^{n+m^2}$ be the closed basic semi-algebraic set defined by the constraints (\ref{cons:first}-\ref{cons:3.9}). Consider the infinite-dimensional optimization problem
$$
\mathcal{P}_\lambda: \quad \widehat{\rho}_\lambda = \min_{x,w}\left\{ \int_{\overline{\mathbf{K}}} p_\lambda d\mu: \int_{\overline{\mathbf{K}}} q_\lambda d\mu = 1,
 \mu \in M(\overline{\mathbf{K}})\right\},
$$
being $p_{\lambda}$, $q_{\lambda}\in \mathbb{R}[x,w]$ as defined above.
Then $\rho_\lambda =\widehat{\rho}_\lambda$.
\end{proposition}
\begin{proof}
It follows by applying Lemma \ref{lemm-frac} to the reformulation of \eqref{omrp} with  the objective function written using $p_{\lambda}$ and $q_{\lambda}$ in \eqref{eq:PyQ}.
\end{proof}

The reader may note the great generality of this class of problems. Depending on the choice of the polynomial weights $\lambda$ we get different classes of problems. Among then, we emphasize the important instances given by:
\begin{enumerate}
\item  $\lambda=(1,0,\ldots,0,0)$ which corresponds to minimize the maximum of a finite number of rational functions,
\item $\lambda=(1,\stackrel{(k)}{\ldots},1,0,\dots,0)$ which corresponds to minimize the sum of the $k$-largest rational functions ($k$-centrum)
\item $\lambda=(0,\stackrel{(k_1)}{\ldots},0,1,\ldots,1,0,\stackrel{(k_2)}{\ldots},0)$ which models the minimization of the $(k_1,k_2)$-trimmed mean of $m$ rational functions,...
\item $\lambda=(1,\alpha,\ldots,\alpha)$ which corresponds to the $\alpha$-centdian, i.e. minimizing the convex combination of the sum and the maximum of the set of rational functions.
\item $\lambda=(1,\ldots,-1)$ which corresponds to minimize the range of a set of rational functions.
\end{enumerate}

\begin{remark}
Problem \ref{pro:omrp1} can be easily extended to deal with the minimization of the ordered median function of a finite number of other ordered median of rational functions. The reader may observe that this can be done by performing a similar transformation to the one in \eqref{omrp} and thus lifting the original problem into a higher dimension space.
\end{remark}

\subsection{Some remarkable special cases} \mbox{\null}

The above general analysis extends the general theory of Lasserre to the case of ordered weighted averages of rational functions. Notice that this approach goes beyond a trivial adaptation of that theory since ordered weighted averages of rational functions are not, in general, neither rational functions nor the supremum of rational functions so that current results cannot be applied to handle these problems. However, one can transform the problem into a new problem embedded in a higher dimension space where it admits a representation that can be cast in the minimization of another rational function in a convenient closed semi-algebraic set. Needless to say that the number of indeterminates  increases with respect to the original one. This may become a problem in particular implementations due to the current state of semidefinite solvers.

In some important particular cases that have been extensively been considered in the field of Operations Research the above approach can be further simplified as we will show in the following. One of this cases, the minimization of the maximum of  finitely many rational functions, has been already analyzed by Laraki and Lasserre \cite{laraki-lasserre}. We will show that the approach in \cite{laraki-lasserre} is also a particular case of the analysis that we present in the following.

For the rest of this subsection we will restrict ourselves, for the sake of readability, to the case of scalar (real) lambda weights. We will begin with the case of $\lambda=(1,\stackrel{(k)}{\ldots},1,0\ldots,0)$, for $1\le k\le m$. Note that for the case $k=1$ we will recover the case analyzed in \cite{laraki-lasserre}, the case $k=m$ is trivial since it reduces to minimize the overall sum and the remaining cases are not yet known.

We are interested in finding the minimum of the sum of the $k$-largest values $\{f_1(x),\ldots,f_m(x)\}$ for all $x\in \mathbf{K}$, being a closed basic semi-algebraic set. In other words, for any $k,\; k=1,\ldots,m-1$, we wish to solve the problem:
\begin{equation*}%\label{pro:kcentrum}
\varrho:=\min_{x\in \mathbf{K}} S_k(x):=\sum_{j=1}^k f_{(j)}(x).
\end{equation*}
We observe that for a given $x$, we have:
$$ S_k(x)=\sum_{j=1}^k f_{(j)}(x)= \max \{ \sum_{j=1}^m v_j f_j(x): \sum_{j=1}^m v_j=k, 0\le v_j\le 1,\; \forall j\}.$$
Therefore, by duality in linear programming:
$$S_k(x)=\min\{ kt+\sum_{j=1}^m r_j: t+r_j\ge f_j(x), \; r_j\ge 0,\;  \forall j\}.$$

Finally, we consider the problem:
\begin{align}
\hat{\varrho} := & \min    \; kt+\sum_{j=1}^m r_j \nonumber\\
&s.t.  \; \nonumber
   t+r_j\ge f_j(x), \quad j=1,\ldots,m \label{pro:kcentrum2}\tag{${\rm kC}$}\\
& \qquad r_j\ge 0,\quad j=1,\ldots,m, \nonumber \\
  &  \qquad x\in \mathbf{K}. \nonumber
\end{align}

Let us denote by $\overline{\mathbf{K}}$ the basic closed semi-algebraic set that defines the feasible region of \eqref{pro:kcentrum2}.
\begin{lemma}
If $\mathbf{K}\subset \mathbb{R}^n$ satisfies Putinar's property then $\overline{\mathbf{K}}\subset \mathbb{R}^{n+m+1}$ satisfies Putinar's property. Moreover $\varrho = \hat{\varrho}$.
\end{lemma}
\begin{proof}
Since we have assumed $\mathbf{K}$ to be compact, for any $j=1,\ldots,m$, there exist $LB_j$, $UB_j$ such that for any $x\in \mathbf{K}$,
$$ LB_j\le f_j(x)\le UB_j.$$
Let us denote $LB=\min_{j=1..m} LB_j$ and $UB=\max_{j=1..m} UB_j$.
Consider an arbitrary $k$, $1\le k\le m-1$ and an arbitrary (but fixed) $\bar x\in \mathbf{K}$. Without loss of generality, assume that $f_m(\bar x) \ge \ldots \ge f_1(\bar x)$. We define the function $$g(t):=\min\{ kt+\sum_{j=1}^m r_j: t+r_j\ge f_j(\bar x), \; r_j\ge 0,\; \forall \; j=1,..,m\}.$$
Clearly, $g$ is piecewise linear and convex; and it attains its minimum on any point of the interval $I_k=(f_{k+1}(\bar x),f_k(\bar x)]$. Indeed, observe that for any $t\in I_k$, the slope of $g$ (i.e. its derivative with respect to $t$) is null since:
$$
g(t)=kt+\sum_{j=1}^k (f_j(\bar x)-t)= \sum_{j=1}^k f_j(\bar x)=S_k(\bar x).
$$
>From the above, we observe that
$$ \varrho =\min_{x\in \mathbf{K}} S_k(x)=\min_{x\in \mathbf{K}} \min \{g(t):kt+\sum_{j=1}^m r_j: t+r_j\ge f_j(x), \; r_j\ge 0,\; \forall \; j=1,..,m\}=\hat{\varrho} .$$

It remains to prove that $\overline {\mathbf{K}}$, the feasible region of problem \eqref{pro:kcentrum2}, satisfies Putinar's condition. First, we observe from the argument above that in order to obtain the minimum value of the function $g$, for any $k=1,..,m-1$ and any $x\in \mathbf{K}$, we only need to consider the range $t \in (f_{(m)}(x),f_{(1)}(x)]$. Hence, the overall range for $t$ can be restricted to $LB\le t \le UB$. On the other hand, for any $x\in \mathbf{K}$, the constraints $0\le r_j\le f_j(x)-t$ set the range of the variable $r_j$. Hence
$$ 0\le r_j\le UB_j-LB,\quad \forall \; j=1,\ldots,m.$$
Including the constraints, $LB\le t \le UB$, $ 0\le r_j\le UB_j-LB,\quad \forall \; j=1,\ldots,m$, in the definition of $\overline{\mathbf{K}}$ does not change the value of $\hat{\varrho}$ and makes the feasible set compact. Thus, satisfying Putinar's condition.
\end{proof}

This approach extends also to the more general case of non-increasing monotone lambda-weights, i.e. $\lambda_1\ge \lambda_2 \ge ...\ge \lambda_m\ge \lambda_{m+1}:=0$ (Note that we define an artificial $\lambda_{m+1}$ to be equal to $0$).  In this case the problem to be solved is:
$$
\varrho_{\lambda}:=\min_{x\in \mathbf{K}} MOM_{\lambda}(x):=\sum_{j=1}^m \lambda_j f_{(j)}(x).
$$
We observe that for a fixed $x\in \mathbf{K}$, we can write the objective function as:
$$ MOM_{\lambda}(x)=\sum_{j=1}^m (\lambda_j-\lambda_{j+1}) S_j(x).$$
Then, we introduce the problem
\begin{eqnarray}
\hat{\varrho_{\lambda}} := & \min & \sum_{k=1}^m (\lambda_k-\lambda_{k+1}) S_k(x) \label{pro:lam-mon2} \\
& &  t_k+r_{kj}\ge f_j(x), \quad j,k=1,\ldots,m, \nonumber \\
& & r_{kj}\ge 0,\quad j,k=1,\ldots,m, \nonumber \\
& & x\in \mathbf{K}. \nonumber
\end{eqnarray}

Let us denote by $\overline{\mathbf{K}}$ the basic closed semi-algebraic set that defines the feasible region of the Problem \eqref{pro:lam-mon2}.
Now, based in the previous lemma, it is straightforward to check the following result.
\begin{lemma}
If $\mathbf{K}\subset \mathbb{R}^n$ satisfies Putinar's property then $\overline{\mathbf{K}}\subset \mathbb{R}^{n+m^2+m}$ satisfies  Putinar's property. Moreover $\varrho_{\lambda} = \hat{\varrho_{\lambda}}$.
\end{lemma}

Another class of problems that can also be analyzed giving rise to a more compact formulation that the one in the general approach \eqref{omrp} is the trimmed mean problem. A trimmed mean objective appears for $\lambda=(\overbrace{0, \ldots, 0}^{k_1}, 1, \ldots, 1,  \overbrace{0, \ldots, 0}^{k_2})$.

This family of problems has attracted a lot of attention in last times in the field of location analysis because of its connections to robust solution concepts. Its rationale rests on the trimmed mean concepts in statistics where the extreme observations (\textit{outliers}) are removed to compute the central estimates (\textit{mean}) of a sample. Thus, we are looking for a point $x^*$ that minimizes the sum of the central functions, once we have excluded the $k_2$ smallest and the $k_1$ largest. Formally, the problem is:
$$ \varrho = \min_{x\in \mathbb{R}^n} \sum_{i=k_1+1}^{n-k_2} f_{(i)}(x).$$

Now, we observe that $\sum_{i=k_1+1}^{n-k_2} f_{(i)}(x)=S_{n-k_2}(x)-S_{k_1}(x)$. Therefore, using the above transformation we have:
\begin{eqnarray*}
S_{k_1}(x)=&= & \max \{ \sum_{j=1}^m v_j f_j(x): \sum_{j=1}^m v_j=k_1, 0\le v_j\le 1,\; \forall j\}, \label{eq:sk1}\\
S_{n-k_2}(x)&=&\min\{ (n-k_2)t+\sum_{j=1}^m r_j: t+r_j\ge f_j(x), \; r_j\ge 0,\;  \forall j\}. \label{eq:sk2}
\end{eqnarray*}
Thus, using both reformulations the trim-mean problem results in:
\begin{align}
\hat{\varrho} := &  \min \, (n-k_2)t+\sum_{j=1}^m r_j -\sum_{j=1}^m v_j f_j(x) \nonumber\\
&s.t.\;
  \sum_{j=1}^m v_j=k_1, \nonumber \\
 & \qquad  t+r_j\ge f_j(x), \quad j=1,\ldots,m, \label{pro:ktrimmean} \tag{${\rm kTr}$} \\
 & \qquad r_j\ge 0,\quad j=1,\ldots,m, \nonumber \\
 & \qquad v_j(v_j-1)=0, \quad j=1,\ldots,m, \nonumber \\
 & \qquad x\in \mathbf{K}. \nonumber
\end{align}

Let us denote by $\overline{\mathbf{K}}$ the basic closed semi-algebraic set that defines the feasible region of  \eqref{pro:ktrimmean}.
\begin{lemma}
If $\mathbf{K}\subset \mathbb{R}^n$ satisfies Putinar's property then $\overline{\mathbf{K}}\subset \mathbb{R}^{n+m+1}$ satisfies  Putinar's property. Moreover $\varrho = \hat{\varrho}$.
\end{lemma}

\begin{remark} We observe that the special formulations for k-centrum \eqref{pro:kcentrum2} and trim-mean \eqref{pro:ktrimmean} are specially suitable for handling these two classes of problems. First of all, we note that if $k_1=0$ the problem reduces to a $k_2$-centrum, variables $v_j$ are not needed and formulation \eqref{pro:ktrimmean} simplifies exactly to   \eqref{pro:ktrimmean}. Second, we point out that both formulations take advantage of the special structure of the considered problems and thus they are simpler than the general formulation \eqref{omrp} applied to these problems. Actually, the number of variables in \eqref{pro:kcentrum2}, for solving the k-centrum problem (resp. \eqref{pro:ktrimmean} for solving the trim-mean problem), is $m+n+1$ (resp. $2m+d+1$) while the number of variables for the same problem using \eqref{omrp} is $m^2+n$. This reduction is remarkable due to the current status of SDP solvers which are not at  a professional level. In spite of that, those problems, where no special structure is known or it cannot  be exploited, can also be tackled using the general formulation  \eqref{omrp} at the price of using larger number of variables.
\end{remark}

\subsection{A convergence result of semidefinite relaxations `$\grave{\mbox{a}}$ la Lasserre'} \mbox{\null}

We are now in position to define the hierarchy of semidefinite relaxations
for solving the $\mrf$ problem. Let $\mathbf{y}=(y_\alpha)$ be a real sequence
indexed in the monomial basis $(x^\beta w^\gamma)$ of $\mathbb{R}[x,w]$ (with $%
\alpha=(\beta,\gamma)\in\mathbb{N}^n\times\N^{m^2}$).  Let $p_\lambda(x,w)$ and $q_\lambda(x,w)$ be defined as in \eqref{eq:PyQ}.

Let $h_{0}(x,w):=p_{\lambda}(x,w)$, and denote $\xi_j:=\lceil (\mathrm{deg}\, g_j)/2\rceil$, $\nu_{j}:=\lceil(\mathrm{deg}%
\,h_{j})/2\rceil$  and $\nu'_{j}:=\lceil(\mathrm{deg}%
\,h'_{j})/2\rceil$  where $\{g_1,\ldots,g_{\ell}\}$ are the polynomial constraints that define $\mathbf{K}$ and $\{h_1, \ldots, h_{m}\}$  and  $\{h'_1, \ldots, h'_{m}\}$ are, respectively,   the polynomial constraints  (\ref{cons:3.4}) and (\ref{const:3.5}) in $\mathbf{\overline K}\setminus \mathbf{K}$, respectively.

Let us denote by $I(0)=\{1,\ldots,n\}$ and $I(j)=\{(j,k)\}_{k=1,\ldots,m}$, for all $j=1,\ldots,m$. With $x(I(0))$, $w(I(j))$ we refer, respectively, to the monomials $x$, $w$ indexed only by subsets of elements in the sets $I(0)$ and $I(j)$, respectively. Then, for $g_k$, with $k=1,\ldots,\ell$, let $\Mo_r(y,I(0))$ (respectively $\Mo_r(g_{k}y,I(0))$) be the moment (resp. localizing) submatrix obtained from $\Mo_r(y)$ (resp. $\Mo_r(g_ky)$) retaining only those rows and columns indexed in the canonical basis of $\mathbb{R}[x(I(0))]$ (resp. $\mathbb{R}[x(I(0))]$). Analogously, for $h_j$ and $h'_j$, $j=1,\ldots,m$, as defined in (\ref{cons:3.4}) and (\ref{const:3.5}), respectively, let $\Mo_r(y,I(0)\cup I(j)\cup I(j+1))$ (respectively $\Mo_r(h_{j}y,I(0)\cup I(j)\cup I(j+1))$, $\Mo_r(h'_{j}y,I(0)\cup I(j)\cup I(j+1))$ ) be the moment (resp. localizing) submatrix obtained from $\Mo_r(y)$ (resp. $\Mo_r(h_jy)$, $\Mo_r(h'_jy)$) retaining only those rows and columns indexed in the canonical basis of $\mathbb{R}[x(I(0))\cup w(I(j))\cup w(I(j+1))]$ (resp. $\mathbb{R}[x(I(0))\cup w(I(j))\cup w(I(j+1))]$).

For $r\geq \max\{ r_{0},\nu_0\}$ where
$r_0:=\displaystyle \max_{k=1,\ldots,\ell} \xi_k$,  $\nu_0:=\displaystyle \max \{ \max_{j=1,\ldots ,m}\nu_{j},  \max_{j=1,\ldots ,m}\nu'_{j}\}$, we introduce the following hierarchy of semidefinite programs:
\begin{equation}\label{lower0}\tag{$\mathbf{Q}_{r}$}
\begin{array}{lll}
\displaystyle\min_{\mathbf{y}} & \L_{\mathbf{y}}(p_{\lambda}) &  \\
\mathrm{s.t.} & \Mo_{r}(\mathbf{y},I(0)) & \succeq 0, \\
& \Mo_{r-\xi_k}(g_k\mathbf{y},I(0)) & \succeq 0,\quad k=1,\ldots ,\ell, \\
& \Mo_r(\mathbf{y},I(0)\cup I(j)\cup I(j+1))& \succeq 0,\quad j=1,\ldots ,m,\\
& \Mo_{r-\nu_{j}}(h_{j}\mathbf{y},I(0)\cup I(j)\cup I(j+1))& \succeq 0,\quad j=1,\ldots ,m,\\
& \Mo_{r-\nu'_{j}}(h'_{j}\mathbf{y},I(0)\cup I(j)\cup I(j+1))& \succeq 0,\quad j=1,\ldots ,m, \\
& \L_y(\sum_{i=1}^m w_{ij}-1)& = 0, \quad j=1,\ldots ,m, \\
& \L_y(\sum_{j=1}^m w_{ij}-1)& = 0, \quad i=1,\ldots ,m, \\
& \L_y(w_{ij}^2-w_{ij})& =0, \quad i,j=1,\ldots,m,\\
& \L_{y}(q_{\lambda}) & =1,
\end{array}%
\end{equation}%
with optimal value denoted $\inf \mathbf{Q}_{r}$ (and $\min \mathbf{Q}_{r}$
if the infimum is attained).

\begin{theorem}
\label{thmain} Let $\mathbf{\overline K}\subset\mathbb{R}^{n+m^2}$ (compact) be the feasible domain of \eqref{omrp}.
%Let $\mathbf{Q}_r$ be
%the semidefinite program (\ref{lower0}) with $(g_k),\; (h_j),\; (h'_j) \subset\mathbb{R}[x,w]$ the polynomial functions defining the constraints of $\overline{\mathbf{K}}$. Then:
Then, with the notation above:

\textrm{(a)} $\inf\mathbf{Q}_r\uparrow \rho_{\lambda}$ as $r\to\infty$.

\textrm{(b)} Let $\mathbf{y}^r,$ be an optimal solution of the SDP relaxation (\ref{lower0}). If
\begin{eqnarray}  \label{finiteconv1}
\mathrm{rank}\,\Mo_r(\mathbf{y}^r,I(0))&= & \mathrm{rank}\,\Mo_{r-r_0}(\mathbf{y}%
^r,I(0))\nonumber\\ {\small \hspace*{-0.75cm}
\mathrm{rank}\,\Mo_r(\mathbf{y}^r,I(0)\cup I(j)\cup I(j+1))} & = & {\small \mathrm{rank}\,\Mo_{r-\nu_0}(\mathbf{y}%
^r,I(0)\cup I(j)\cup I(j+1)) \; j=1,\ldots,m \label{finiteconv2}}
\end{eqnarray}
and if $\mathrm{rank}(\Mo_r(y^*,I(0)\cup (I(k)\cup I(k+1))\cap (I(j)\cup I(j+1))))=1$ for all $j\neq k$
then $\min\Q_r=\rho_{\lambda}$.

Moreover, let $\Delta_j:=\{(x^*(j),w^*(j))\}$ be the set of solutions obtained by the condition (\ref{finiteconv2}). Then, every $(x^*,w^*)$ such that $(x^*_i,w^*_i)_{i\in I(j)}=(x^*(j),w^*(j))$ for some $\Delta_j$ is an optimal solution of Problem $\mrf$.
\end{theorem}
\begin{proof}
The convergence of the semidefinite relaxation (\ref{lower0}) was proved by Jibetean and De Klerk \cite{jibetean} for a general rational function over a closed semialgebraic set. Here, we use that result applied to the rational function in (\ref{eq:pdivq}). Moreover, the index set of the indeterminates in the feasible set given by constraints (\ref{cons:first})-(\ref{cons:3.9}) admits the decomposition  $I(k)$, $k=0\ldots,m$ that satisfies the running intersection property (see \cite[(1.3)]{sparse}) and therefore, the result follows by combining Theorem 3.2 in \cite{sparse} and the results in \cite{jibetean}.
\end{proof}

The above theorem allows us to approximate and solve the original problem $\mrf$ up to any degree of accuracy by solving block diagonal (sparse) SDP programs which are convex programs for each fixed relaxation order $r$ and that can be solved with available open source solvers as SeDuMi, SDPA, SDPT3 \cite{sdpt3}, etc.

\section{Generalized Location Problems with rational objective functions \label{s:locomf}}

This sections considers a wide family of continuous location problems that has attracted a lot of attention  in the recent literature of location analysis but for which there are not common solution approaches. The challenge is to design a common resolution approach  to solve them for different distances and in any finite dimension.

We are given a set $A=\{a_1,\ldots,a_n\}\subset \mathbb{R}^d$ endowed with an $\ell_{\tau}$-norm (here $\ell_\tau$ stands for the norm $\|x\|_{\tau} = \left(\sum_{i=1}^d |x_i|^\tau\right)^{\frac{1}{\tau}}$, for all $x\in \R^d$); and a feasible domain $\mathbf{K}\subset \mathbb{R}^d$, closed and semi-algebraic. The goal is to find a point $x^*\in \mathbf{K}\subset \mathbb{R}^d$ minimizing some globalizing function of the distances to  the set $A$. Here, we consider that the globalizing function is rather general and that it is given as a rational function.

Some well-known examples are listed below (see e.g. \cite{BC2009}, \cite{drezner2007}, \cite{EMPR09}, \cite{LMP06} or \cite{NP05}) :

\begin{itemize}
\item $f(u_1,\ldots,u_n)={ \displaystyle \sum_{i<j}^n |u_i-u_j|}$, absolute deviation or envy problem.

\item $f(u_1,\ldots,u_n)={ \displaystyle \sum_{i=1}^n (u_i-\bar u)^2}$, variance problem.

\item $f(u_1,\ldots,u_n)={ \displaystyle \sum_{j=1}^n \frac{w_j}{u_j^2}}$, obnoxious facility location.

\item $f(u_1,\ldots,u_n)={ \displaystyle \sum_{j=1}^n \frac{b_j}{1+h_j |u_j|^{\lambda}}}$, Huff competitive location.
\end{itemize}
The main feature and what distinguishes location problem from other general purpose optimization problems, is that the dependence of the decision variables is given throughout the norms to the demand points in $A$, i.e. $\|x-a_i\|_{\tau}$.
In this section, we consider a generalized version of continuous single facility location problems with rational objective functions over closed semi-algebraic feasible sets.

Let $f_j(u):=\frac{p_j(u)}{q_j(u)}: \mathbb{R}^n\mapsto \mathbb{R}$, $j=1,\ldots,m$ be rational functions with $q_j(u)>0$ for all $j$. We shall define the dependence of $f_j$ to the decision variable $x\in \mathbb{R}^d$ via $u=(u_1,\ldots,u_n)$, where $u_i:\mathbb{R}^d\mapsto \mathbb{R}$, $u_i(x):=\|x-a_i\|_{\tau}$, $i=1,\ldots,n$. Therefore, the $j$-th component of the ordered  median objective function of our problems reads as:
$$\begin{array}{llll} \tilde f_j(x):&\mathbb{R}^d &\mapsto &\mathbb{R} \\
& x &\mapsto & \tilde f_j(x):=f_j(\|x-a_1\|_{\tau},\ldots,\|x- a_n\|_{\tau}).
\end{array} $$

Consider the following problem:
\begin{equation} \label{pro:LocPop}\tag{$\mathbf{LOCOMRF}$}
\rho_\lambda \,:=\displaystyle\min_{x}\{\dsum_{j=1}^m \lambda_j(x) \tilde f_{(j)}(x): x\in \mathbf{K}\,\},
\end{equation}%
where:

$\bullet$  $\K \subseteq \R^n$ satisfies Putinar's property,

$\bullet$ $\tau:=\frac{r}{s}$, $r,s\in \mathbb{N}$, $r\ge s$ and $gcd(r,s)=1$.

This problem does not reduce to the family $\mrf$ considered above since the dependence on the decision variable $x$ is not given in the form of polynomials. Note that $\ell_{\tau}$-norms are not, in general, polynomials.

To avoid this inconvenience,  we introduce the following auxiliary problem.

\begin{eqnarray}
 \quad \overline{\rho}_\lambda = &\displaystyle \min_{x,w,u,v} & \sum_{j=1}^{m}\lambda_{j}(x)\sum_{i=1}^{m} f_{i}(u)w_{ij} \label{pro:locgen}\\
 &\mbox{s.t. } &\sum_{j=1}^{m}w_{ij} = 1, \text{ for } i=1, \ldots, m,\nonumber\\
 & &\sum_{i=1}^{m}w_{ij} = 1, \text{ for } j=1, \ldots, m,\nonumber\\
 & &\sum_{i=1}^{m} w_{ij} f_{i}(u) \geq  \sum_{i=1}^{m} w_{i j+1} f_{i}(u), j=1, \ldots, m,\nonumber\\%label{pro:omrp3}\\
 & & w_{ij}^2 - w_{ij} = 0, \text{ for } i,j=1, \ldots, m, \nonumber\\
 & & v_{kl}^s\ge (x_k-a_{kl})^r,\; k=1,\ldots,n,\; l=1,\dots,d, \label{c:vmayor} \\
  & & v_{kl}^s\ge (a_{kl}-x_l)^r,\; k=1,\ldots,n,\; l=1,\dots,d, \label{c:vmenor} \\
  & & u_k^r=(\sum_{l=1}^d v_{kl})^s,\; k=1,\ldots,n, \nonumber\\%\label{c:z}
%  & & u_i=z_i^s,\; i=1,\ldots,n,\nonumber \\
 & &\sum_{i, j=1}^m w^2_{ij} \leq m, \nonumber\\
 & &w_{ij} \in \R, \; \forall\ i,j=1,\ldots,m, \nonumber\\
 & & v_{kl} \in \R, u_k \in \R, k=1, \ldots, n, l=1, \ldots, d, \nonumber \\
 & &x\in \mathbf{K}.\nonumber
\end{eqnarray}

%Denote by $v_{ij}^s=|x_j-a_{ij}|$, for $i=1, \ldots, m$ and $j=1, \ldots, n$, by $z_i^r=\sum_{j=1}^n v_{ij}^r$ and $w_{ij}$ as in \eqref{omp}. Then, $(\mathbf{P})$ is equivalent to

We note in passing that the above problem simplifies for those cases where $r$ is even. In these cases, we can replace the two sets of constraints, namely
(\ref{c:vmayor}) and (\ref{c:vmenor}) by the simplest constraint
$$
v_{kl}^s=(x_k-a_{kl})^r, \quad \forall\; k,l.
$$
This reformulation reduces by $(n\times d)$ the number of constraints defining the feasible set. Moreover, these constraints do not induce semidefinite constraints in the moment approach but linear matrix inequalities which are easier to handle.  Following the same scheme of the proof  in Theorem \ref{th:feas} we get the following result, whose proof is left to the reader.

\begin{theorem}
Let $x$ be a feasible solution of (\ref{pro:LocPop}) then there exists a solution $(x,u,v,w)$ for (\ref{pro:locgen}) such that their objective values are equal. Conversely, if $(x,u,v,w)$ is a feasible solution for (\ref{pro:locgen}) then there exists a solution $(x)$ for (\ref{pro:LocPop}) having the same objective value. In  particular $\varrho_{\lambda}=\overline{\varrho}_{\lambda}$. Moreover, if $\mathbf{K}\subset \mathbb{R}^d$ satisfies Putinar's property then $\overline{\mathbf{K}}\subset \mathbb{R}^{d+m^2+n(d+2)}$ also satisfies Putinar's property.
\end{theorem}
%\begin{proof}
%The proof is similar and follows the scheme of that in Theorem \ref{th:feas} and therefore it is left to the reader.
%\end{proof}

Now, we can prove a convergence result that allows us to solve, up to any degree of accuracy, the above class of problems. Let $\mathbf{y}=(y_\alpha)$ be a real sequence
indexed in the monomial basis $(x^\beta u^\gamma v^\delta w^\zeta)$ of $\mathbb{R}[x,u,v,w]$ (with $%
\alpha=(\beta,\gamma,\delta,\zeta)\in\mathbb{N}^d\times \mathbb{N}^n\times \mathbb{N}^{nd}\times \mathbb{N}^{m^2}$).

Let $h_{0}(x,u,v,w):=p_{\lambda}(x,u,v,w)$, and denote $\xi_j:=\lceil (\mathrm{deg}\, g_j)/2\rceil$ and $\nu_{j}:=\lceil(\mathrm{deg}%
\,h_{j})/2\rceil$, where $\{g_1,\ldots,g_{\ell}\}$, and $\{h_1, \ldots, h_{3m+m^2+2n(d+1)+1}\}$ are, respectively,  the polynomial constraints that define $\mathbf{K}$ and $\mathbf{\overline K}\setminus \mathbf{K}$ in (\ref{pro:locgen}). For $r\geq r_{0}:=%
\displaystyle\max \{\max_{k=1,\ldots,\ell} \xi_k,$   $\displaystyle \max_{j=0,\ldots ,\l+3m+m^2+1}\nu_{j}\}$, introduce the hierarchy of semidefinite
programs:
\begin{equation}\label{lower}\tag{$\mathbf{Q}_{r}$}
\begin{array}{lll}
\displaystyle\min_{\mathbf{y}} & \L_{\mathbf{y}}(p_{\lambda}) &  \\
\mathrm{s.t.} & \Mo_{r}(\mathbf{y}) & \succeq 0, \\
& \Mo_{r-\xi_k}(g_k,\mathbf{y}) & \succeq 0,\quad k=1,\ldots ,\ell, \\
& \Mo_{r-\nu_{j}}(h_{j},\mathbf{y}) & \succeq 0,\quad j=1,\ldots ,3m+m^2+1, \\
& \L_{y}(q_{\lambda}) & =1,%
\end{array}
\end{equation}%
with optimal value denoted $\inf \mathbf{Q}_{r}$ (and $\min \mathbf{Q}_{r}$
if the infimum is attained).

\begin{theorem}
\label{thmain} Let $\mathbf{\overline K}\subset\mathbb{R}^{d+m^2+n(d+2)}$ (compact) be the feasible domain of Problem (\ref{pro:locgen}). Let $\mathbf{Q}_r$ be
the semidefinite program (\ref{lower}). %with $(g_k),\; (h_j)\subset\mathbb{R}[x,u,v,w]$ the polynomial functions defining the constraints of $\overline{\mathbf{K}}$. Then:
Then, with the notation above:

\textrm{(a)} $\inf\mathbf{Q}_r\uparrow \rho_{\lambda}$ as $r\to\infty$.

\textrm{(b)} Let $\mathbf{y}^r$ be an optimal solution of the SDP relaxation
$\mathbf{Q}_r$ in (\ref{lower}). If
\begin{equation*}  \label{finiteconv}
\mathrm{rank}\,\Mo_r(\mathbf{y}^r)\,=\,\mathrm{rank}\,\Mo_{r-r_0}(\mathbf{y}%
^r)\,=\,t
\end{equation*}
then $\min\Q_r=\rho_{\lambda}$ and one may extract $t$ points $(x^*(k),u^*(k),v^*(k),w^*(k))_{k=1}^t\subset\mathbf{\overline K}$, all global minimizers of the $\mrf$ problem.
\end{theorem}
\begin{proof}
The convergence of the semidefinite relaxation \ref{lower} was proved by Jibetean and De Klerk \cite{jibetean} for a general rational function over a closed semialgebraic set. Here, we apply this result applied to the rational function in (\ref{eq:pdivq}) and therefore, the result follows.
\end{proof}

Here, we also observe that one can exploit the block diagonal structure of the problem since there is a sparsity pattern in the variables of formulation \eqref{pro:locgen}. The reader may note that the only monomials that appear in that formulation are of the form $x^{\alpha}u_i^{\beta}\prod_{j=1}^mv_{ij}^{\gamma_j}$ for all $i=1,\ldots,m$. Hence, a result similar to Theorem \ref{thmain} also holds for the hierarchy $(Q_r)$ of SDP applied to the location problem. Nevertheless, although we have used it in our computational test, we do not give specific details for the sake of presentation and  because of the similarity with Theorem \ref{thmain}.

\begin{ex}\label{ex:2}
We illustrate the above results with an instance of the well-known Weber problem with $\ell_3$-norm and for $20$ random demand points in $\R^3$. Let   {\small $A=\{(0.0758,0.0540,0.5308)$,
$(0.7792,0.9340,0.1299)$, $(0.5688,0.4694,0.0119)$, $(0.3371,0.1622,0.7943)$, $(0.3112,0.5285,0.1656)$, \\ $(0.6020,0.2630,0.6541)$,  $(0.6892,0.7482,0.4505)$, $(0.0838,0.2290,0.9133)$, $(0.1524,0.8259,0.5383)$, \\ $(0.9961,0.0782,0.4427)$, $(0.1066,0.9619,0.0046)$, $(0.7749,0.8173,0.8687)$, $(0.0844,0.3998,0.2599)$, \\ $(0.8000,0.4314,0.9106)$, $(0.1818,0.2638,0.1455)$, $(0.1361,0.8693,0.5797)$, $(0.5499,0.1450,0.8530)$, \\ $(0.5499,0.1450,0.8530)$, $(0.4018,0.0760,0.2399)$, $(0.1233,0.1839,0.2400)\}$.}

\noindent Then, the problem consists of
\begin{eqnarray*}
\min & \dsum_{a \in A}  \|x-a\|_3\\
 s.t.&x\in  \R^3.
\end{eqnarray*}
The feasible region of the first SDP relaxation of this problem, which in this case is $r=2$, contains $20$ moment matrices of size $36\times 36$, $160$ localizing matrices of size $8 \times 8$ and 36 equality constraints. The exact optimal solution is given by $\overline{x}=(0.426397,0.438730,0.455857)$ with optimal value $\overline{f}=8.729976$.  We get with our approach, using  SDPT3\cite{sdpt3}, an optimal solution  $x^{*}=(0.426397,0.438730,0.455857)$, for the first relaxation of the problem with optimal value $f^{*}=8.729976$. Thus, the relative error is $\overline{\epsilon} = \frac{|f^{*} - \overline{f}|}{\overline{f}} =2.199595\times 10^{-13}$.

For the same set of points, we consider a modification of the above problem by adding an extra nonconvex constraint:
\begin{eqnarray*}
\min & \dsum_{a \in A}  \|x-a\|_3\\
 s.t.& x_1^2 - 2x_2^2-2x_3^2 \geq 0,\\
 &x\in  \R^3.
\end{eqnarray*}
The exact optimal solution of this problem is $\widetilde{x}=(0.562304,0.266296,0.295262 )$ with optimal value $\widetilde{f}=10.109333$.  The reader may note that the original  solution $\bar{x}$ is not feasible for the new problem. Using our approach, again for the first relaxation order, we get $x^{**}=(0.562304,0.266296,0.295262)$ with optimal value $f^{**}=10.109333$. Hence, the relative error in this case is $\widetilde{\epsilon} = \frac{|f^{**} - \widetilde{f}|}{\widetilde{f}} =5.801151\times 10^{-9}$.

We show in Figure \ref{fig20} the feasible region of our problem as well as the demand points and the optimal solutions  (the exact and the ones obtained with our relaxed formulations) of the problems. The demand points in $A$ are represented by ' $*$', the optimal solution, $x^*$, of the SDP relaxation without the nonconvex constraint by ' {\tiny$\blacksquare$}' and the optimal solution, $x^{**}$, of the SDP relaxation with the nonconvex constraint is depicted by ' $\bullet$'.

\begin{figure}[h]
\centering
\includegraphics[scale=0.3]{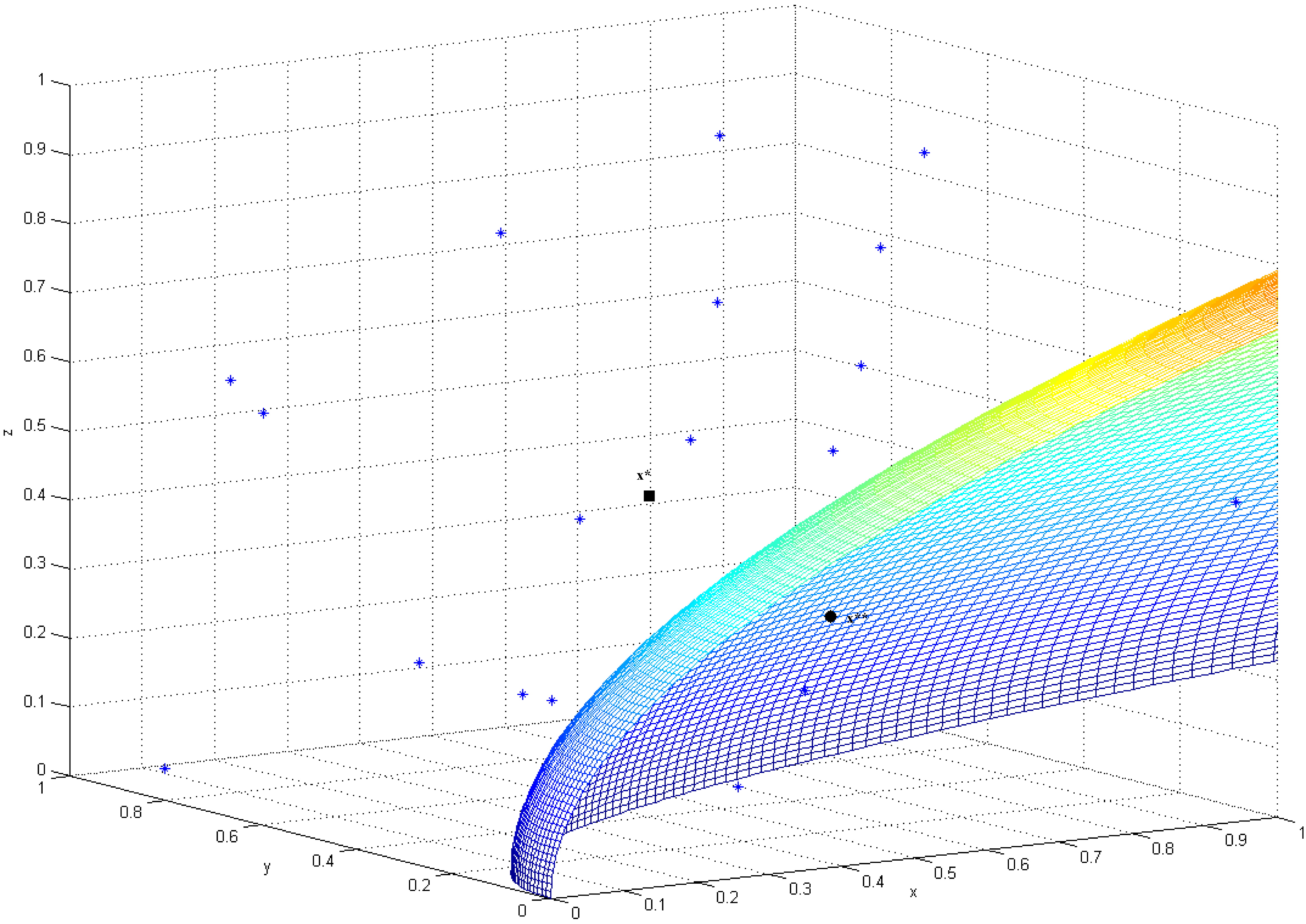}
\caption{Feasible region, demand points and optimal solutions of Example \ref{ex:2}.\label{fig20}}
\end{figure}

\end{ex}

In the following, we will apply this general methodology to get the reformulation of the most standard problems in Location Analysis (see Nickel and Puerto \cite{NP05}) that will be later the basis of our computational experiments: minisum (Weber) and minimax (center), $k$-centrum, $(k_1,k_2)$-trimmed mean and range problems.

\subsection{Weber or median  problem\label{s:proweber}}
In the standard version of the Weber problem, we are given a set of demand points $\{a_1,\ldots,a_n\}$ in $\mathbb{R}^d$ and a set of non-negative weights $\omega_1,\ldots,\omega_n$ and one looks for a point $x^*$ minimizing the weighted Euclidean distance from the demand point. In other words, the problem is:
$$ \min_{x\in \mathbb{R}^d} \sum_{i=1}^n \omega_i \|x-a_i\|_2.$$
This problem has been largely studied in the literature of Location Analysis and perhaps its most well-known algorithm is the so called Weiszfeld algorithm (see \cite{weiszfeld}). This problem is a convex one and Weiszfeld algorithm is a gradient type iterative algorithmic scheme for which several convergence results are known.

Here, we observe that this problem corresponds to a very particular choice of the elements in \eqref{pro:LocPop}: $\lambda=(1, \ldots, 1)$, $f_i(u)=\omega_iu$ and $r=2$, $s=1$. Furthermore, the general formulation (\ref{pro:LocPop}) simplifies since there is no actual sorting. Therefore, we can avoid many of our instrumental variables, namely, the problem can be cast into the form:

\begin{align}
\min& \sum_{i=1}^{n} \omega_i z_{i} \nonumber\\
 s.t.& \;z_i^2=\sum_{j=1}^d (x_j-a_{ij})^2, i=1, \ldots, n, \nonumber\\
& \dsum_{j=1}^d x_j^2  + z_i^2 \leq M, i=1, \ldots, n, \label{wp}\tag{$\mathbf{WP}$}\\
 & z_i\ge 0, \; i=1, \ldots, n,\nonumber\\
& x\in \R^d.\nonumber
\end{align}

\subsection{The minimax or center problem\label{s:procenter}}
The minimax location problem looks for the location of a server $x\in \mathbb{R}^d$ that minimizes the maximum weighted distance to a given set of demands points $\{a_1,\ldots,a_n\}$ in $\mathbb{R}^d$. Formally, the problem can be stated as:
$$ \min_{x\in \mathbb{R}^d} \max_{i=1,\ldots,n} \omega_i \|x-a_i\|_2,
$$
for some weights $\omega_1, \ldots, \omega_n \geq 0$.

Once more, this problem has been extensively analyzed in the literature of Location Analysis and the most well-known algorithms to solve it are those by Elzinga-Hearn (only valid in $\mathbb{R}^2$ with Euclidean distance) and Dyer \cite{D1,D2} and Megiddo \cite{M4}  which are polynomial in fixed dimension. Again, we observe that this problem corresponds to a very particular choice of the elements in \eqref{pro:LocPop}: $\lambda=(1,0, \ldots,0)$, $f_i(u)=\omega_iu$ and $r=2$, $s=1$. In this case, the general formulation (\ref{pro:LocPop}) simplifies and therefore, we can avoid many of our instrumental variables, namely, the problem can be formulated as:

\begin{align}
 \min & \mbox{ } t  \nonumber\\
s.t. & \; z_i^2=\sum_{j=1}^d (x_j-a_{ij})^2,\; i=1, \ldots, n,\nonumber \\
 &  \omega_iz_i\le t,\; i=1, \ldots, n,\nonumber \\
 & \sum_{j=1}^d x_j^2+z_i^2+t^2 \le M,\; i=1, \ldots, n,\label{pro:center}\tag{$\mathbf{CP}$}\\
 & t,z_i,  \ge 0,\quad i=1,\ldots, n, \nonumber \\
& x\in \mathbf{K}. \nonumber
\end{align}

\subsection{The k-centrum problem\label{s:prokcentrum}}

The $k$-centrum location problem consists of finding the point $x^*$ that minimizes the sum of the $k$ largest distances with respect to a given set of demands points $\{a_1,\ldots,a_n\}$ in $\mathbb{R}^d$.
Formally, the problem can be stated as:
$$ \min_{x\in \mathbb{R}^d} \max_{i=1,\ldots,k} d_{(i)}(x),
$$
where $d_{(i)}(x)=\|x-a_{\sigma(i)}\|$ for a permutation $\sigma$ such that $d_{\sigma(1)}(x)\ge \ldots \ge d_{\sigma(n)}(x)$.
This problem has been considered in several papers and textbooks (see \cite{NP05}, \cite{DH02}). Currently, there exist few approaches to solve it in the plane (i.e. $d=2$) and with the Euclidean norm that do not extend further to higher dimension nor other norms (see \cite{DN09-01,DN09-02,ERChD10}).
The objective function of this problem is described by a vector of $\lambda$-parameters $\lambda=(\overbrace{1, \ldots, 1}^{k},0, \ldots,0)$, $f_i(u)=u$, $r=2$, $s=1$. Using the result in the reformulation \eqref{pro:kcentrum2} the problem can be restated as:

\begin{align}
\hat{\varrho} := & \min \;  kt+\sum_{i=1}^n r_i  \nonumber\\
& s.t. \; z_i^2=\sum_{j=1}^d (x_j-a_{ij})^2,\; i=1, \ldots, n,\nonumber \\
&\qquad  t+r_i\ge z_i,\; i=1, \ldots, n,\label{pro:kcentrum}\tag{$\mathbf{kCP}$}\\
& \qquad  \sum_{j=1}^d x_j^2+z_i^2+r_i^2 \le M,\; i=1, \ldots, n,\nonumber\\
& \qquad t,r_i,z_i  \ge 0,\quad i=1,\ldots, n, \nonumber \\
& \qquad x\in \mathbb{R}^d. \nonumber
\end{align}

\subsection{The $(k_1,k_2)$-trimmed-mean problem\label{s:protrim}}

The $(k_1,k_2)$-trimmed-mean location problem looks for a point $x^*$ that minimizes the sum of the central distances, once we have excluded the $k_2$ closest and the $k_1$ furthest. Formally, the problem is:
$$ \min_{x\in \mathbb{R}^d} \sum_{i=k_1+1}^{n-k_2} d_{(i)}(x),$$
where $d_{(i)}(x)=\|x-a_{\sigma(i)}\|_2$ for a permutation $\sigma$ such that $d_{\sigma(1)}(x)\ge \ldots \ge d_{\sigma(n)}(x)$.
This problem has been considered in several papers and textbooks (see \cite{NP05}, \cite{DH02}). Currently, there exists two approaches to solve it in the plane (i.e. $d=2$) and with the Euclidean norm that do not extend further to higher dimension nor other norms (see \cite{DN09-01,DN09-02}).
The objective function of this problem, in terms of the elements in \eqref{pro:LocPop}, is described by a vector of $\lambda$-parameters $\lambda=(\overbrace{0, \ldots, 0}^{k_1}, 1, \ldots, 1,  \overbrace{0, \ldots, 0}^{k_2})$, $f_i(u)=u$, $r=2$, $s=1$. Here, we could apply the general formulation derived from \eqref{pro:LocPop}. Nevertheless, that approach needs many decision variables which affects the sizes of the problems to be handled. Rather than the general formulation, we present here an alternative problem, based on (\ref{pro:ktrimmean}), which takes advantage of the particular structure  of this problem and reduces the number of variables needed for its representation.

We consider the problem:
\begin{align}
\min & \; (n-k_2)t+\sum_{i=1}^n r_i - \sum_{i=1}^{n} u_i z_i \nonumber \\
s.t. &\; z_i^2=\sum_{j=1}^d (x_j-a_{ij})^2,\; i=1,\ldots,n,\nonumber \\
& \sum_{i=1}^n u_i= k_1,\nonumber\\
& u_i(u_i-1)=0, \quad i=1,\ldots,n, \nonumber \\
&  t+r_i\ge z_i, \quad i=1,\ldots,n, \tag{$\mathbf{TMP}$}\label{pro:trimmed}\\
& \sum_{j=1}^d x_j^2+z_i^2+t^2+u_i^2+r_i^2 \le M,\; i=1,\ldots,n, \nonumber\\
& z_i, r_i, u_i, t \ge 0,\quad i=1,\ldots,n, \nonumber \\
&  x\in \mathbb{R}^d. \nonumber
\end{align}

\subsection{The range problem\label{s:prorange}}

The last problem that we address in our computational experiments is the range location problem. This problem consists of minimizing the difference (range) between the maximum and minimum distances with respect to  a given set of demands points $\{a_1,\ldots,a_n\}$ in $\mathbb{R}^d$ (see \cite{DN09-01,DN09-02,NP05}).
Formally, the problem can be stated as:
$$ \min_{x\in \mathbb{R}^d} \left[\max_{i=1,\ldots,n} \|x-a_i\|_2-\min_{i=1,\ldots,n} \|x-a_i\|_2 \right].
$$
This problem corresponds to the following choice of the elements in \eqref{pro:LocPop}: $\lambda=(1,0,\ldots,0,-1),\; f_i(u)=u$ and $r=2$, $s=1$. A simplified reformulation of the problem reduces to:

\begin{align}
\min &\mbox{ } z-t  \nonumber \\
s.t. & \;z_i^2=\sum_{j=1}^d (x_j-a_{ij})^2,\; i=1,\ldots,n,\nonumber \\
&  t\le z_i\le z,\; i=1,\ldots,n,\label{pro:range}\tag{$\mathbf{RP}$}\\
 & \sum_{j=1}^d x_j^2+z_i^2+t^2+z^2 \le M,\; i=1,\ldots,n,\nonumber \\
 & t,z,z_i  \ge 0,\quad i=1,\ldots,n, \nonumber \\
 & x\in \mathbb{R}^d. \nonumber
\end{align}

\section{Computational Experiments}
A series of computational experiments have been performed in order
to evaluate the behavior of the proposed methodology.
Programs have been coded in MATLAB R2010b and executed in a PC with an
Intel Core i7 processor at 2x 2.93 GHz and 8 GB of RAM. The semidefinite programs have been solved by calling SDPT3 4.0\cite{sdpt3}.

We run the algorithm for several well-known continuous location problems: Weber problem, center problem, k-center problem, trimmed-mean problem and range problem. For each of them, we obtain the CPU times for computing solutions as well as the gap with respect to the optimal solution obtained with the battery of functions in \verb"optimset" of MATLAB or the implementation by \cite{DN09-01,DN09-02}.

With regard to computing the accuracy of an obtained solution, we use the following measure for the error (see \cite{waki}):
\begin{equation} \label{eq:error}
\epsilon_{\rm obj} = \dfrac{|\text{the optimal value of the SDP } - {\rm fopt}|}{\max\{1, {\rm fopt}\}},
\end{equation}
where ${\rm fopt}$ is the optimal objective value for the problem obtained with the functions in  \verb"optimset" or the implementation by \cite{DN09-01,DN09-02}.

We have organized our computational experiments in five different problems types that coincide with those described previously in sections \ref{s:proweber}-\ref{s:prorange}. Our test problems are generated to be comparable with previous results of some algorithms in the plane but, in addition, we also consider problems in $\mathbb{R}^3$.  Thus, we report on randomly generated points on the unit square and in the unit cube. Depending on the problem, we have been able to solve different problem sizes. In all problems, we could solve instances with at least 500 points for planar and 3-dimensional problems % or ?? points in the space
 and with an average accuracy higher than $10^{-5}$. (We remark that for instance we could solve instances of more than 1000 points for Weber and center problems with high precisions.)

Our goal is to present the results organized per problem type, framework space ($\mathbb{R}^2$ or $\mathbb{R}^3$) and relaxation order. We report for Weber  problem on the first two relaxations to show that raising relaxation order one gains some extra precision (as expected) at the price of higher CPU times. In spite of that, the considered problems seems to be very well-approximated even with the first relaxation (as shown by our results).  For this reason, we only report results for relaxation order $r=2$ for the remaining problem types, namely center, $k$-centrum, range and trim-mean.

The results in our tables, for each size and problem type, are the average of ten runs. In all cases our tables are organized in the same way. Rows give the results for the different number of demand points considered in the problems. Column $\texttt{n}$ stands for the number of points considered in the problem, $\texttt{CPU time}$ is the average running time needed to solve each of the instances, $\epsilon_{obj}$ gives the error measure (see \ref{eq:error}). The final block of  3 columns informs on the sizes of the SDP problems to be solved: $\texttt{\#Cols}$, $\texttt{\#Rows}$ and $\texttt{\%NonZero}$ represent, respectively, the number of columns, rows and the percentage of nonzero entries of the constraint matrices of the problems to be considered.

We tested problems with up to $500$ demands points  (except for Weber problem where we considered $1000$ demands points) randomly generated in the unit square and the unit cube. We move $\texttt{n}$ between $10$ and $500$ (or $1000$ for Weber problem) and ten instances were generated for each value of $\texttt{n}$. The first  relaxation of the problems was solved in all cases. For the k-centrum problem type we considered three different $k$ values to test the difficulty of problems with respect to that parameter, $k=\lceil0.1 \texttt{n}\rceil, \lceil0.5 \texttt{n}\rceil, \lceil0.9  \texttt{n}\rceil$ (tables  \ref{table:4} and \ref{table:5}).

Tables \ref{table:1}-\ref{table:6} show the  averages CPU times and gaps obtained. Table \ref{table:1} summarizes the results of the Weber problems. We remark that problems with up to 1000 demand points on the plane  are solved with the first relaxation in few seconds and with accuracy higher than $10^{-4}$. Raising the relaxation order, we improve accuracy till $10^{-6}$ at the cost of multiplying CPU time by a factor of 8. Table \ref{table:2} refers to Weber problem in the $3d$ space. Results are similar although precision is higher when considering the second relaxation order. Table \ref{table:3} reports the results for the center problem on the plane and the $3d$-space. CPU times are slightly larger than for the Weber problem but accuracy are also better specially for sizes up to 100 demand points. Tables \ref{table:4} and \ref{table:5} are devoted to show the behavior of our approach for three different $k$ values of the $k$-centrum problem (Table \ref{table:4} in $\mathbb{R}^2$ and Table \ref{table:5} in $\mathbb{R}^3$). We observe that for small values of $k$, i.e. $k=\lceil 0.1n\rceil \mbox{ or } \lceil 0.5n\rceil$ the $k$-centrum is slightly harder than for values closer to $n$. The remaining factors behave similarly to those in Weber or center problems. Table \ref{table:7} reports the results for the range problem. The behavior of these problems is similar to that of the $k$centrum problems both in CPU time and accuracy. Finally, Table \ref{table:6} summarizes the results for the trimmed-mean problems. These are the harder problems among the five considered problem types. We are able to solve similar sizes with similar accuracies using the first order relaxation. However, CPU times are significantly higher than for the other problem types. These results show that this methodology can be efficiently applied to solve medium to large sized location problems.

>From our tables we conclude that Weber problem is the simplest one whereas the trimmed-mean problem  is the hardest one, as expected. We remark that CPU times increase linearly with the number of points in all problem types. A linear regression between these times and the number of points gives a regression coefficient $R$-squared (coefficient of determination of the regression)   greater than $0.98$ for all the problems. Therefore, this shows a linear dependence, up to the tested sizes, between problem sizes and CPU times for solving the corresponding relaxations. Observe that the sizes of the matrices in the SDP relaxations increase exponentially with the number of points. Nevertheless, the percentage of nonzero elements in the constraint matrices decreases very slowly (hyperbolically) when increasing the size (number of points) of the problems.

\section{Conclusions}
We develop a unified tool for minimizing weighted ordered averaging of rational functions. This approach goes beyond a trivial adaptation of the  general theory of moments-sos since ordered weighted averages of rational functions are not, in general, neither rational functions nor the supremum of rational functions so that current results cannot directly be applied to handle these problems. As an important application we cast a general class of continuous location problems within the minimization of OWA rational functions. We report computational results that show the powerfulness of this methodology to solve medium to large continuous location problems.

This new approach solves a broad class of convex and non convex continuous location problems that, up to date, were only partially solved in the specialized literature. We have tested this methodology with some medium to large size standard ordered median location problems in the plane and in the 3-dimensional space. Our goal was not to compete with previous algorithms since most of them are either problem specific or only applicable for planar problems. However, in all cases we obtained reasonable CPU times and high accuracy results even with first relaxation order. Our good results heavily rely on the fact that we have detected sparsity patterns in these problems reducing considerably the sizes of the SDP object to be considered.

The  two main lines for further research in this area would be to increase both the sizes and the classes of problems efficiently solved. These goals may be achieved by improving the efficiency of available SDP solvers and/or by finding alternative formulations that take advantage of new sparsity and symmetry patterns.

%\section*{Acknowledgements}
%The authors were partially supported by the project FQM-5849 (Junta de Andaluc\'ia$\backslash$FEDER). The first and third author were partially supported by the project  MTM2010-19576-C02-01 (MICINN, Spain). The first author was also supported by the Juan de la Cierva grant JCI-2009-03896.

\begin{landscape}
\vspace*{1cm}
\begin{table}[h]

\begin{tabular}{|c|cc|ccc||cc|ccc|}\hline   & \multicolumn{5}{|c||}{First Relaxation ($r=2$)} &  \multicolumn{5}{c|}{Second Relaxation  ($r=3$)} \\ \hline
 \texttt{n}    &  \texttt{CPU time} & $\epsilon_{\rm obj}$  &    \texttt{\#Cols}   &  \texttt{\#Rows} & \texttt{\%NonZero}  &  \texttt{CPU time} & $\epsilon_{\rm obj}$  &    \texttt{\#Cols}   &  \texttt{\#Rows}  & \texttt{\%NonZero} \\\hline
     10    & 0.63  & 0.00191774 & 1420  & 214   & 0.780\% & 2.45  & 0.00008689 & 6200  & 587   & 0.279\% \\
    20    & 1.03  & 0.00079178 & 2840  & 414   & 0.403\% & 5.67  & 0.00002648 & 12400 & 1147  & 0.143\% \\
    30    & 1.03  & 0.00062061 & 4260  & 614   & 0.272\% & 8.94  & 0.00002065 & 18600 & 1707  & 0.096\% \\
    40    & 1.57  & 0.00082654 & 5680  & 814   & 0.205\% & 11.43 & 0.00000992 & 24800 & 2267  & 0.072\% \\
    50    & 2.12  & 0.00015842 & 7100  & 1014  & 0.165\% & 13.29 & 0.00000269 & 31000 & 2827  & 0.058\% \\\hline
    60    & 2.31  & 0.00027699 & 8520  & 1214  & 0.137\% & 16.95 & 0.00000213 & 37200 & 3387  & 0.048\% \\
    70    & 2.72  & 0.00044228 & 9940  & 1414  & 0.118\% & 20.54 & 0.00000434 & 43400 & 3947  & 0.042\% \\
    80    & 3.03  & 0.00044249 & 11360 & 1614  & 0.103\% & 26.98 & 0.00000243 & 49600 & 4507  & 0.036\% \\
    90    & 3.38  & 0.00031839 & 12780 & 1814  & 0.092\% & 29.20 & 0.00000194 & 55800 & 5067  & 0.032\% \\
    100   & 3.92  & 0.00027367 & 14200 & 2014  & 0.083\% & 31.57 & 0.00000174 & 62000 & 5627  & 0.029\% \\\hline
    150   & 6.12  & 0.00027644 & 21300 & 3014  & 0.055\% & 46.31 & 0.00000555 & 93000 & 8427  & 0.019\% \\
    200   & 8.36  & 0.00021865 & 28400 & 4014  & 0.042\% & 65.75 & 0.00000190 & 124000 & 11227 & 0.015\% \\
    250   & 10.42 & 0.00028088 & 35500 & 5014  & 0.033\% & 87.13 & 0.00000656 & 155000 & 14027 & 0.012\% \\
    300   & 12.19 & 0.00019673 & 42600 & 6014  & 0.028\% & 102.95 & 0.00001241 & 186000 & 16827 & 0.010\% \\
    350   & 14.63 & 0.00018747 & 49700 & 7014  & 0.024\% & 124.36 & 0.00000850 & 217000 & 19627 & 0.008\% \\
    400   & 17.25 & 0.00021381 & 56800 & 8014  & 0.021\% & 145.62 & 0.00000333 & 248000 & 22427 & 0.007\% \\
    450   & 20.37 & 0.00007970 & 63900 & 9014  & 0.019\% & 167.02 & 0.00000476 & 279000 & 25227 & 0.007\% \\
    500   & 22.03 & 0.00011803 & 71000 & 10014 & 0.017\% & 187.02 & 0.00000754 & 310000 & 28027 & 0.006\% \\\hline
    600   & 28.11 & 0.00012725 & 85200 & 12014 & 0.014\% & 232.19 & 0.00000287 & 372000 & 33627 & 0.005\% \\
    700   & 33.47 & 0.00015215 & 99400 & 14014 & 0.012\% & 274.88 & 0.00000332 & 434000 & 39227 & 0.004\% \\
    800   & 39.50 & 0.00009879 & 113600 & 16014 & 0.010\% & 334.10 & 0.00000420 & 496000 & 44827 & 0.004\% \\
    900   & 45.31 & 0.00011740 & 127800 & 18014 & 0.009\% & 389.00 & 0.00000350 & 558000 & 50427 & 0.003\% \\
    1000  & 55.68 & 0.00012513 & 142000 & 20014 & 0.008\% & 443.13 & 0.00000351 & 620000 & 56027 & 0.003\% \\
 \hline
\multicolumn{11}{c}{}
\end{tabular}

    \caption{Computational results for planar Weber problem and first and second relaxation.\label{table:1}}
    \end{table}
\newpage

\vspace*{2cm}
\begin{table}[h]
\begin{tabular}{|c|cc|ccc||cc|ccc|}\hline  & \multicolumn{5}{|c||}{First Relaxation ($r=2$)} &  \multicolumn{5}{c|}{Second Relaxation  ($r=3$)} \\ \hline
 \texttt{n}   &   \texttt{CPU time} & $\epsilon_{\rm obj}$  &    \texttt{\#Cols}   &  \texttt{\#Rows} &  \texttt{\%NonZero}  &    \texttt{CPU time} & $\epsilon_{\rm obj}$  &    \texttt{\#Cols}   &  \texttt{\#Rows} &  \texttt{\%NonZero} \\\hline
    10    & 1.19  & 0.00112213 & 2900  & 384   & 0.442\% & 9.13  & 0.00000379 & 17100 & 1343  & 0.124\% \\
    20    & 1.84  & 0.00036619 & 5800  & 734   & 0.231\% & 23.89 & 0.00000000 & 34200 & 2603  & 0.064\% \\
    30    & 2.56  & 0.00019790 & 8700  & 1084  & 0.157\% & 28.97 & 0.00000000 & 51300 & 3863  & 0.043\% \\
    40    & 3.54  & 0.00011433 & 11600 & 1434  & 0.118\% & 45.19 & 0.00000000 & 68400 & 5123  & 0.033\% \\
    50    & 4.27  & 0.00008446 & 14500 & 1784  & 0.095\% & 58.34 & 0.00000001 & 85500 & 6383  & 0.026\% \\\hline
    60    & 5.04  & 0.00019406 & 17400 & 2134  & 0.080\% & 66.09 & 0.00000000 & 102600 & 7643  & 0.022\% \\
    70    & 6.23  & 0.00009027 & 20300 & 2484  & 0.068\% & 77.67 & 0.00000000 & 119700 & 8903  & 0.019\% \\
    80    & 7.09  & 0.00018689 & 23200 & 2834  & 0.060\% & 90.86 & 0.00000000 & 136800 & 10163 & 0.016\% \\
    90    & 8.01  & 0.00010943 & 26100 & 3184  & 0.053\% & 124.89 & 0.00000000 & 153900 & 11423 & 0.015\% \\
    100   & 9.87  & 0.00005552 & 29000 & 3534  & 0.048\% & 164.37 & 0.00000008 & 171000 & 12683 & 0.013\% \\\hline
    150   & 14.16 & 0.00004856 & 43500 & 5284  & 0.032\% & 211.02 & 0.00000000 & 256500 & 18983 & 0.009\% \\
    200   & 20.33 & 0.00003049 & 58000 & 7034  & 0.024\% & 275.02 & 0.00000000 & 342000 & 25283 & 0.007\% \\
    250   & 25.97 & 0.00005964 & 72500 & 8784  & 0.019\% & 429.67 & 0.00000014 & 427500 & 31583 & 0.005\% \\
    300   & 34.00 & 0.00004677 & 87000 & 10534 & 0.016\% & 501.09 & 0.00000006 & 513000 & 37883 & 0.004\% \\
    350   & 39.82 & 0.00004154 & 101500 & 12284 & 0.014\% & 588.29 & 0.00000007 & 598500 & 44183 & 0.004\% \\
    400   & 47.27 & 0.00005233 & 116000 & 14034 & 0.012\% & 746.70 & 0.00000011 & 684000 & 50483 & 0.003\% \\
    450   & 57.08 & 0.00003325 & 130500 & 15784 & 0.011\% & 762.54 & 0.00000000 & 769500 & 56783 & 0.003\% \\
    500   & 65.93 & 0.00002952 & 145000 & 17534 & 0.010\% & 1063.50& 0.00000000 & 855000 & 63083 & 0.003\% \\\hline
    \multicolumn{11}{c}{}
\end{tabular}
    \caption{Computational results for Weber problem in $\R^3$ and first and second relaxation.\label{table:2}}
    \end{table}

\newpage

\vspace*{2cm}

\begin{table}[h]
\begin{tabular}{|c|cc|ccc||cc|ccc|}\hline  & \multicolumn{5}{|c||}{$\R^2$} &  \multicolumn{5}{c|}{$\R^3$} \\ \hline
 \texttt{n}  &  \texttt{CPU time} & $\epsilon_{\rm obj}$  &    \texttt{\#Cols}   &  \texttt{\#Rows} & \texttt{\%NonZero}  &  \texttt{CPU time} & $\epsilon_{\rm obj}$  &    \texttt{\#Cols}   &  \texttt{\#Rows} & \texttt{\%NonZero} \\\hline
    10    & 0.95  & 0.00000002 & 3150  & 384   & 0.423\% & 1.90  & 0.00000001 & 5700  & 629   & 0.259\% \\
    20    & 1.78  & 0.00000001 & 6300  & 734   & 0.221\% & 4.05  & 0.00000000 & 11400 & 1189  & 0.137\% \\
    30    & 2.68  & 0.00000001 & 9450  & 1084  & 0.150\% & 6.24  & 0.00000008 & 17100 & 1749  & 0.093\% \\
    40    & 3.78  & 0.00000001 & 12600 & 1434  & 0.113\% & 8.96  & 0.00000000 & 22800 & 2309  & 0.071\% \\
    50    & 4.68  & 0.00000000 & 15750 & 1784  & 0.091\% & 12.05 & 0.00000000 & 28500 & 2869  & 0.057\% \\\hline
    60    & 6.05  & 0.00000000 & 18900 & 2134  & 0.076\% & 16.63 & 0.00000000 & 34200 & 3429  & 0.048\% \\
    70    & 8.48  & 0.00000000 & 22050 & 2484  & 0.065\% & 18.84 & 0.00000002 & 39900 & 3989  & 0.041\% \\
    80    & 10.28 & 0.00000002 & 25200 & 2834  & 0.057\% & 28.08 & 0.00000000 & 45600 & 4549  & 0.036\% \\
    90    & 13.60 & 0.00000005 & 28350 & 3184  & 0.051\% & 32.16 & 0.00000000 & 51300 & 5109  & 0.032\% \\
    100   & 18.86 & 0.00000005 & 31500 & 3534  & 0.046\% & 38.78 & 0.00000291 & 57000 & 5669  & 0.029\% \\\hline
    150   & 31.12 & 0.00002157 & 47250 & 5284  & 0.031\% & 59.19 & 0.00006902 & 85500 & 8469  & 0.019\% \\
    200   & 38.76 & 0.00013507 & 63000 & 7034  & 0.023\% & 82.01 & 0.00011298 & 114000 & 11269 & 0.014\% \\
    250   & 44.34 & 0.00027776 & 78750 & 8784  & 0.019\% & 111.64 & 0.00013810 & 142500 & 14069 & 0.012\% \\
    300   & 58.10 & 0.00033715 & 94500 & 10534 & 0.015\% & 124.47 & 0.00030316 & 171000 & 16869 & 0.010\% \\
    350   & 81.59 & 0.00047225 & 110250 & 12284 & 0.013\% & 170.43 & 0.00043926 & 199500 & 19669 & 0.008\% \\
    400   & 90.22 & 0.00048347 & 126000 & 14034 & 0.012\% & 172.05 & 0.00052552 & 228000 & 22469 & 0.007\% \\
    450   & 93.50 & 0.00047479 & 141750 & 15784 & 0.010\% & 242.66 & 0.00057288 & 256500 & 25269 & 0.006\% \\
    500   & 151.64 & 0.00066416 & 157500 & 17534 & 0.009\% & 226.73 & 0.00059268 & 285000 & 28069 & 0.006\% \\
 \hline \multicolumn{11}{c}{}
\end{tabular}
\caption{Computational results for center problem in $\R^2$ and $\R^3$ and first relaxations.\label{table:3}}
\end{table}

\newpage

\vspace*{2cm}

\begin{table}[h]
\begin{tabular}{|c|cc|cc|cc|ccc|}\hline  & \multicolumn{2}{|c||}{$k=\lceil 0.1\,n\rceil$} & \multicolumn{2}{|c||}{$k=\lceil 0.5\,n\rceil$} & \multicolumn{2}{|c||}{$k=\lceil 0.9\,n\rceil$} & \multicolumn{3}{|c|}{Sizes} \\ \hline
 \texttt{n}   &     \texttt{CPU time} & $\epsilon_{\rm obj}$  &  \texttt{CPU time} & $\epsilon_{\rm obj}$  &     \texttt{CPU time} & $\epsilon_{\rm obj}$  &    \texttt{\#Cols}   &  \texttt{\#Rows} &  \texttt{\#NonZero}\\\hline
    10    & 2.64   & 0.00000630 & 2.76   & 0.00000081 & 2.59   & 0.00017665 & 6570   & 944   & 0.175\% \\
    20    & 6.43   & 0.00001375 & 6.15   & 0.00000298 & 5.30   & 0.00000545 & 13140  & 1854  & 0.089\% \\
    30    & 10.88  & 0.00000379 & 9.89   & 0.00000410 & 9.16   & 0.00000102 & 19710  & 2764  & 0.060\% \\
    40    & 15.89  & 0.00000717 & 16.33  & 0.00000090 & 12.22  & 0.00000122 & 26280  & 3674  & 0.045\% \\
    50    & 21.24  & 0.00000282 & 18.51  & 0.00000083 & 16.77  & 0.00000105 & 32850  & 4584  & 0.036\% \\\hline
    60    & 25.77  & 0.00000077 & 25.41  & 0.00000283 & 20.21  & 0.00000806 & 39420  & 5494  & 0.030\% \\
    70    & 28.01  & 0.00000204 & 31.02  & 0.00000234 & 25.07  & 0.00000192 & 45990  & 6404  & 0.026\% \\
    80    & 37.25  & 0.00000085 & 31.48  & 0.00000044 & 30.66  & 0.00000220 & 52560  & 7314  & 0.023\% \\
    90    & 47.16  & 0.00000062 & 41.07  & 0.00000765 & 33.92  & 0.00000086 & 59130  & 8224  & 0.020\% \\
    100   & 53.68  & 0.00000084 & 41.42  & 0.00000065 & 39.49  & 0.00000188 & 65700  & 9134  & 0.018\% \\\hline
    150   & 86.48  & 0.00000089 & 68.48  & 0.00000056 & 65.95  & 0.00000059 & 98550  & 13684 & 0.012\% \\
    200   & 123.02 & 0.00000056 & 96.40  & 0.00000075 & 88.10  & 0.00000275 & 131400 & 18234 & 0.009\% \\
    250   & 149.26 & 0.00003681 & 135.67 & 0.00000071 & 113.68 & 0.00000161 & 164250 & 22784 & 0.007\% \\
    300   & 180.38 & 0.00000408 & 161.84 & 0.00000081 & 146.22 & 0.00000349 & 197100 & 27334 & 0.006\% \\
    350   & 223.27 & 0.00003013 & 193.31 & 0.00003623 & 176.46 & 0.00000151 & 229950 & 31884 & 0.005\% \\
    400   & 260.27 & 0.00000079 & 225.07 & 0.00003689 & 201.01 & 0.00000376 & 262800 & 36434 & 0.005\% \\
    450   & 290.23 & 0.00004512 & 272.55 & 0.00000097 & 237.23 & 0.00000168 & 295650 & 40984 & 0.004\% \\
    500   & 345.93 & 0.00000224 & 310.19 & 0.00000119 & 269.99 & 0.00000200 & 328500 & 45534 & 0.004\% \\
\hline
    \multicolumn{10}{c}{}
\end{tabular}
    \caption{Computational results for planar $k$-centrum problems and first relaxation ($r=2$).\label{table:4}}
\end{table}

\newpage

\vspace*{2cm}

\begin{table}[h]
\begin{tabular}{|c|cc|cc|cc|ccc|}\hline  & \multicolumn{2}{|c||}{$k=\lceil 0.1\,n\rceil$} & \multicolumn{2}{|c||}{$k=\lceil 0.5\,n\rceil$} & \multicolumn{2}{|c||}{$k=\lceil 0.9\,n\rceil$} & \multicolumn{3}{|c|}{Sizes} \\ \hline
 \texttt{n}   &     \texttt{CPU time} & $\epsilon_{\rm obj}$  &  \texttt{CPU time} & $\epsilon_{\rm obj}$  &     \texttt{CPU time} & $\epsilon_{\rm obj}$  &    \texttt{\#Cols}   &  \texttt{\#Rows} &  \texttt{\%NonZero}\\\hline
    10    & 7.06  & 0.00041340 & 5.85  & 0.00000039 & 6.05  & 0.00000168 & 10780 & 1469  & 0.114\% \\
    20    & 16.40 & 0.00000950 & 15.42 & 0.00000095 & 16.30 & 0.00000019 & 21560 & 2869  & 0.059\% \\
    30    & 27.63 & 0.00001682 & 23.72 & 0.00000028 & 27.12 & 0.00000132 & 32340 & 4269  & 0.039\% \\
    40    & 45.25 & 0.00000075 & 42.31 & 0.00000086 & 37.38 & 0.00000077 & 43120 & 5669  & 0.030\% \\
    50    & 54.39 & 0.00000282 & 53.66 & 0.00000026 & 51.94 & 0.00000087 & 53900 & 7069  & 0.024\% \\\hline
    60    & 63.16 & 0.00000259 & 59.34 & 0.00000091 & 63.91 & 0.00000065 & 64680 & 8469  & 0.020\% \\
    70    & 85.17 & 0.00000144 & 81.32 & 0.00000258 & 74.24 & 0.00000079 & 75460 & 9869  & 0.017\% \\
    80    & 106.65 & 0.00000326 & 83.96 & 0.00000044 & 88.76 & 0.00000158 & 86240 & 11269 & 0.015\% \\
    90    & 114.38 & 0.00000209 & 93.85 & 0.00000100 & 103.56 & 0.00000092 & 97020 & 12669 & 0.013\% \\
    100   & 122.01 & 0.00000088 & 109.17 & 0.00000224 & 118.03 & 0.00000067 & 107800 & 14069 & 0.012\% \\\hline
    150   & 235.10 & 0.00000073 & 211.54 & 0.00000890 & 187.51 & 0.00000135 & 161700 & 21069 & 0.008\% \\
    200   & 305.51 & 0.00002407 & 255.54 & 0.00007106 & 284.80 & 0.00000157 & 215600 & 28069 & 0.006\% \\
    250   & 403.89 & 0.00000519 & 348.32 & 0.00004300 & 357.79 & 0.00000143 & 269500 & 35069 & 0.005\% \\
    300   & 492.04 & 0.00046130 & 433.69 & 0.00007630 & 471.78 & 0.00000174 & 323400 & 42069 & 0.004\% \\
    350   & 529.61 & 0.00041229 & 484.87 & 0.00000058 & 448.60 & 0.00001791 & 377300 & 49069 & 0.003\%\\
    400   & 619.97 & 0.00000091 & 585.93 & 0.00000055 & 523.81 & 0.00000829 & 431200 & 56069 & 0.003\%\\
    450   & 705.99 & 0.00048727 & 693.77 & 0.00000037 & 580.06 & 0.00004327 & 485100 & 63069 & 0.003\%\\
    500   & 817.75 & 0.00012138 & 789.77 & 0.00000087 & 664.94 & 0.00000318 & 539000 & 70069 & 0.002\%\\ \hline
    \multicolumn{10}{c}{}
\end{tabular}
    \caption{Computational results for $k$-centrum problems in $\R^3$ and first relaxation ($r=2$).\label{table:5}}
\end{table}

\newpage

\vspace*{2cm}

\begin{table}[h]
\begin{tabular}{|c|cc|ccc||cc|ccc|}\hline  & \multicolumn{5}{|c||}{$\R^2$} &  \multicolumn{5}{c|}{$\R^3$} \\ \hline
 \texttt{n}   & \texttt{CPU time} & $\epsilon_{\rm obj}$  &    \texttt{\#Cols}   &  \texttt{\#Rows} & \texttt{\%NonZero} &  \texttt{CPU time} & $\epsilon_{\rm obj}$  &    \texttt{\#Cols}   &  \texttt{\#Rows} & \texttt{\%NonZero} \\\hline
    10    & 2.96  & 0.00007519 & 6060  & 629   & 0.252\% & 5.68  & 0.00001997 & 10080 & 965   & 0.164\% \\
    20    & 7.04  & 0.00001750 & 12120 & 1189  & 0.133\% & 18.45 & 0.00015758 & 20160 & 1805  & 0.088\% \\
    30    & 13.94 & 0.00098322 & 18180 & 1749  & 0.091\% & 35.37 & 0.00028187 & 30240 & 2645  & 0.060\% \\
    40    & 14.53 & 0.00002124 & 24240 & 2309  & 0.069\% & 35.77 & 0.00032049 & 40320 & 3485  & 0.045\% \\
    50    & 24.49 & 0.00004314 & 30300 & 2869  & 0.055\% & 65.80 & 0.00051293 & 50400 & 4325  & 0.037\% \\\hline
    60    & 23.49 & 0.00047832 & 36360 & 3429  & 0.046\% & 59.19 & 0.00005082 & 60480 & 5165  & 0.031\% \\
    70    & 34.87 & 0.00003903 & 42420 & 3989  & 0.040\% & 68.46 & 0.00006841 & 70560 & 6005  & 0.026\% \\
    80    & 38.69 & 0.00026693 & 48480 & 4549  & 0.035\% & 79.54 & 0.00003016 & 80640 & 6845  & 0.023\% \\
    90    & 42.34 & 0.00042121 & 54540 & 5109  & 0.031\% & 90.76 & 0.00017468 & 90720 & 7685  & 0.021\% \\
    100   & 58.36 & 0.00052427 & 60600 & 5669  & 0.028\% & 97.26 & 0.00015535 & 100800 & 8525  & 0.019\% \\\hline
    150   & 65.04 & 0.00021457 & 90900 & 8469  & 0.019\% & 159.41 & 0.00094711 & 151200 & 12725 & 0.012\% \\
    200   & 98.23 & 0.00041499 & 121200 & 11269 & 0.014\% & 197.66 & 0.00040517 & 201600 & 16925 & 0.009\% \\
    250   & 131.42 & 0.00033959 & 151500 & 14069 & 0.011\% & 274.14 & 0.00057559 & 252000 & 21125 & 0.007\% \\
    300   & 159.87 & 0.00014556 & 181800 & 16869 & 0.009\% & 322.21 & 0.00036845 & 302400 & 25325 & 0.006\% \\
    350   & 169.29 & 0.00003661 & 212100 & 19669 & 0.008\% & 393.80 & 0.00096204 & 352800 & 29525 & 0.005\% \\
    400   & 167.74 & 0.00123896 & 242400 & 22469 & 0.007\% & 361.12 & 0.00022448 & 403200 & 33725 & 0.005\%\\
    450   & 218.70 & 0.00207328 & 272700 & 25269 & 0.006\% & 513.55 & 0.00044016 & 453600 & 37925 & 0.004\%\\
    500   & 228.68 & 0.00438388 & 303000 & 28069 & 0.006\% & 554.94 & 0.00028013 & 504000 & 42125 & 0.004\%\\
\hline
    \multicolumn{11}{c}{}
\end{tabular}
    \caption{Computational results for range problem in $\R^2$ and $\R^3$ and first relaxation.\label{table:7}}
\end{table}

\newpage

\vspace*{2cm}

\begin{table}[h]
\begin{tabular}{|c|cc|ccc||cc|ccc|}\hline  & \multicolumn{5}{|c||}{$\R^2$} &  \multicolumn{5}{c|}{$\R^3$} \\ \hline
 \texttt{n}   & \texttt{CPU time} & $\epsilon_{\rm obj}$  &    \texttt{\#Cols}   &  \texttt{\#Rows} & \texttt{\%NonZero} &  \texttt{CPU time} & $\epsilon_{\rm obj}$  &    \texttt{\#Cols}   &  \texttt{\#Rows} & \texttt{\%NonZero} \\\hline
    10    & 5.31  & 0.00017041 & 11760 & 1784  & 0.087\% & 14.09 & 0.00000197 & 18080 & 2669  & 0.059\% \\
    20    & 12.39 & 0.00000619 & 23520 & 3534  & 0.044\% & 33.85 & 0.00047792 & 36160 & 5269  & 0.030\% \\
    30    & 18.11 & 0.00020027 & 35280 & 5284  & 0.029\% & 49.16 & 0.00000670 & 54240 & 7869  & 0.020\% \\
    40    & 30.39 & 0.00035248 & 47040 & 7034  & 0.022\% & 73.13 & 0.00001450 & 72320 & 10469 & 0.015\% \\
    50    & 36.04 & 0.00181487 & 58800 & 8784  & 0.018\% & 98.17 & 0.00001624 & 90400 & 13069 & 0.012\% \\\hline
    60    & 49.16 & 0.00085810 & 70560 & 10534 & 0.015\% & 131.38 & 0.00003143 & 108480 & 15669 & 0.010\% \\
    70    & 60.57 & 0.00012995 & 82320 & 12284 & 0.013\% & 161.25 & 0.00004420 & 126560 & 18269 & 0.009\% \\
    80    & 73.54 & 0.00092073 & 94080 & 14034 & 0.011\% & 188.51 & 0.00012265 & 144640 & 20869 & 0.008\% \\
    90    & 76.12 & 0.00040564 & 105840 & 15784 & 0.010\% & 203.06 & 0.00011847 & 162720 & 23469 & 0.007\% \\
    100   & 91.26 & 0.00218668 & 117600 & 17534 & 0.009\% & 220.68 & 0.00011032 & 180800 & 26069 & 0.006\% \\\hline
    150   & 153.31 & 0.00814047 & 176400 & 26284 & 0.006\% & 400.37 & 0.00026203 & 271200 & 39069 & 0.004\% \\
    200   & 257.23 & 0.00032380 & 235200 & 35034 & 0.004\% & 552.19 & 0.00056138 & 361600 & 52069 & 0.003\% \\
    250   & 339.72 & 0.00051519 & 294000 & 43784 & 0.004\% & 659.01 & 0.00046219 & 452000 & 65069 & 0.002\% \\
    300   & 326.52 & 0.00225994 & 352800 & 52534 & 0.003\% & 884.40 & 0.00038481 & 542400 & 78069 & 0.002\% \\
    350   & 410.32 & 0.00047898 & 411600 & 61284 & 0.003\% & 955.53 & 0.00061467 & 632800 & 91069 & 0.002\%\\
    400   & 582.36 & 0.00047130 & 470400 & 70034 & 0.002\% & 1165.79 & 0.00058261 & 723200 & 104069 & 0.002\%\\
    450   & 631.58 & 0.00060180 & 529200 & 78784 & 0.002\% & 1931.76 & 0.00081711 & 813600 & 117069 & 0.001\%\\
    500   & 685.79 & 0.00079679 & 588000 & 87534 & 0.002\% & 9151.90 & 0.00063861 & 904000 & 130069 & 0.001\%\\
\hline
    \multicolumn{11}{c}{}
\end{tabular}
    \caption{Computational results for trimmed mean problem with $k1=k2=\lceil 0.20\,n\rceil$ in $\R^2$ and $\R^3$ and first relaxation.\label{table:6}}
\end{table}

\newpage

\end{landscape}

\end{document}